\documentclass[twoside,a4paper,10pt]{article} 
\usepackage[T1]{fontenc}
\usepackage[utf8]{inputenc} 

\usepackage{booktabs} 
\usepackage{array} 
\usepackage{paralist} 
\usepackage{verbatim} 
\usepackage{subfig} 
\usepackage[backend=bibtex,style=alphabetic,sorting=nyvt]{biblatex}
\usepackage{amssymb}
\usepackage[dvipsnames]{xcolor}
\usepackage{mathtools}
\mathtoolsset{showonlyrefs}
\usepackage[scr=rsfs,cal=boondox]{mathalfa}
\usepackage{amsthm}
\usepackage{tikz-cd}
\usepackage[colorlinks, allcolors=blue]{hyperref}
\usepackage{titlesec}
\usepackage{graphicx} 
\usepackage{calrsfs}
\DeclareMathAlphabet{\pazocal}{OMS}{fszplm}{m}{n}

\usepackage{geometry} 
\usepackage{fancyhdr} 

\titleformat{\section}
       { \centering \bfseries \large}
       {\thesection.}
       {0.5em}
       {}
       []

\titleformat{\subsection}[runin]
       {\normalfont\bfseries}
       {\thesubsection.}
       {0.5em}
       {}
       [.]
\titleformat{\subsubsection}[runin]
       {\normalfont\bfseries}
       {\thesubsubsection.}
       {0.5em}
       {}
       [.]

\usepackage[titles,subfigure]{tocloft} 

\newcommand{\dbar}{\frac{\partial}{\partial \bar{\tau}}}

\newcommand{\Mtil}{\widetilde{M}}

\newcommand{\dds}{\restr{\frac{\partial}{\partial s}}{s=0}}

\newcommand{\dxdy}{\frac{dxdy}{y^2}}
\newcommand{\iotau}{\iota_{u\frac{\partial}{\partial u}}}
\newcommand{\uud}{\underline{d}}

\newcommand{\alphahat}{\widehat{\alpha}}

\newcommand{\8}{\infty}

\newcommand{\Z}{\mathbb{Z}}
\newcommand{\NN}{\mathbb{N}}

\newcommand{\id}{\mathbf{1}}

\newcommand{\vbf}{\mathbf{v}}
\newcommand{\wbf}{\mathbf{w}}

\newcommand{\Qbar}{\overline{\Q}}

\newcommand{\SO}{\mathrm{SO}}

\newcommand{\SL}{\mathrm{SL}}

\newcommand{\KM}{\mathrm{KM}}
\newcommand{\BM}{\mathrm{BM}}

\newcommand{\GL}{\mathrm{GL}}

\newcommand{\Stil}{\widetilde{S}}

\newcommand{\D}{\mathbb{D}}

\newcommand{\HH}{\mathbb{H}}

\newcommand{\R}{\mathbb{R}}
\newcommand{\A}{\mathbb{A}}
\newcommand{\Q}{\mathbb{Q}}
\newcommand{\C}{\mathbb{C}}

\newcommand{\PP}{\mathbb{P}}
\newcommand{\nfrak}{\mathfrak{n}}
\newcommand{\afrak}{\mathfrak{a}}

\newcommand{\dfrak}{\mathfrak{d}}

\newcommand{\g}{\mathfrak{g}}
\newcommand{\cfrak}{\mathfrak{c}}

\newcommand{\vphat}{\widehat{\phi}}

\newcommand{\End}{\mathrm{End}}

\newcommand{\vol}{\mathrm{vol}}

\newcommand{\Tr}{\mathrm{Tr}}
\newcommand{\tr}{\mathrm{tr}}

\newcommand{\N}{\mathrm{N}}
\newcommand{\sgn}{\mathrm{sgn}}
\newcommand{\sing}{\mathrm{sing}}

\newcommand{\reg}{\mathrm{reg}}
\newcommand{\hol}{\mathrm{hol}}
\newcommand{\CT}{\mathrm{CT}_{\rho=0}}

\newcommand{\Ocal}{\mathcal{O}}
\newcommand{\Scal}{\pazocal{S}}

\newcommand{\Fcal}{\pazocal{F}}

\newcommand{\Cl}{\textrm{Cl}}

\newcommand{\Ecal}{\pazocal{E}}

\newcommand{\intx}{\int_{-\frac{1}{2}}^{\frac{1}{2}}}

\newcommand{\re}{\mathrm{Re}}
\newcommand{\supp}{\mathrm{supp}}
\newcommand{\Mat}{\mathrm{Mat}}
\newcommand{\im}{\mathrm{Im}}

\newcommand{\hooklongrightarrow}{\lhook\joinrel\longrightarrow}

\newcommand\restr[2]{{
  \left.\kern-\nulldelimiterspace 
  #1 
  \vphantom{\big|} 
  \right|_{#2} 
  }}
 
\newtheoremstyle{mytheoremstyle} 
    {1.5em}                    
    {1.em}                    
    {\itshape}                   
    {}                           
    {\normalsize \bfseries}                   
    {.}                          
    {0,5em}                       
    {}  
    
\theoremstyle{mytheoremstyle}

\newtheorem{thm}{Theorem}[section]
\newtheorem*{thm*}{Theorem}
\newtheorem{cor}{Corollary}[thm]
\newtheorem*{cor*}{Corollary}
\newtheorem{lem}[thm]{Lemma}
\newtheorem{prop}[thm]{Proposition}

\newtheorem*{que*}{Question}

\theoremstyle{remark}
\newtheorem{rmk}{Remark}[section]
\newtheorem*{ex}{Example}
\numberwithin{equation}{section}

\renewenvironment{abstract}
 {\small
  \begin{center}
  \bfseries \abstractname\vspace{-.5em}\vspace{0pt}
  \end{center}
  \list{}{
    \setlength{\leftmargin}{3cm}%
    \setlength{\rightmargin}{\leftmargin}%
  }%
  \item\relax}
 {\endlist}

\begin{document}
\title{A regularized theta lift on the symmetric space of $\SL_N$}
\author{ Romain Branchereau }
\date{} 
\maketitle

\begin{abstract} We define a regularized lift from harmonic weak Maass forms of weight $2-N$ to differential forms of degree $N-1$ on the symmetric space $\SL_N(\R)/\SO(N)$, that are smooth outside of certain modular symbols. We show that this lift is adjoint to the derivative of a theta lift. We compute the periods of the regularized lift over tori and relate them to Fourier coefficients of Hilbert-Eisenstein series. 
\end{abstract}

{
  \tableofcontents
}

\section{Introduction}

Regularized lifts are theta lifts from harmonic weak Maass forms to functions on orthogonal Grassmannians $O(p,q)/O(p)\times O(q)$ with singularities along smaller sub-Grassmannians. The regularization is done by taking the limit of an integral over a truncated fundamental domain. First introduced by Borcherds in his seminal work \cite{borcherds1,borcherds2}, many variations of these regularized lifts (now often called Borcherds lifts) have been considered since then, for example in \cite{Bruinier12,Garcia16,AS20,FH21,CF23}.  Most notably, Bruinier-Funke \cite{BF04} constructed a lift from weak harmonic Maass forms to differential forms with singularities along special cycles, and which is adjoint to the Kudla-Millson lift. Since the work of Bruinier \cite{bruinierborch,B99} and Bruinier-Funke \cite{BF04}, it is known that these regularized lifts satisfy certain current equations and play an important role in arithmetic geometry. When the signature is $(p,2)$ and the Grassmannian is a hermitian domain, these regularized lifts are Green currents for special cycles. For general signature $(p,q)$, the regularized lift of Bruinier-Funke defines a differential character.

When the regularized lift is a function on the Grassmannian, it can be evaluated at special points associated with CM extensions. There is an extensive literature on these special values of regularized lifts, mainly in the case $O(p,2)$ -- see for example \cite{Schofer,V15,BKY,BY09,BEY,E17,Li22,Li23,BLY25}. This was greatly motivated by the work of Gross-Zagier \cite{gz86} and Gross-Kohnen-Zagier \cite{GKZ87}. In particular, in signature $(2,2)$ these values were conjectured by Gross-Zagier to be logarithms of algebraic numbers, which has been recently proved by Li \cite{Li23} and Bruinier–Li–Yang \cite{BLY25}. An important aspect of these special values is that they are related to Fourier coefficients of derivatives of Eisenstein series; see \cite{BKY,BY09,BEY}. This plays an important role in the proof of Li \cite{Li23} of the conjecture of Gross-Zagier. There have been a few results about special values of regularized lifts in signature $(n,1)$. First, by Bruinier-Schwagenscheidt \cite{BS21}, and more recently by Herrero-Imamo\u{g}lu-Pippich-Schwagenscheidt \cite{HIPS24} in signature $(3,1)$.

In this paper, we consider a regularized lift similar to the construction of Bruinier-Funke \cite{BF04}, lifting weak harmonic Maass forms to singular differential forms on a symmetric space. Instead of an orthogonal symmetric space, we consider the symmetric space $S=\SL_N(\R)/\SO(N)$. 
We show that this lift is adjoint to the derivative of a cohomological lift on $S$. In particular, we replace the evaluation at a special point by the integral over a cycle attached to an algebraic torus, and relate it to the derivative of a Hilbert-Eisenstein series.

\subsection{A cohomological theta lift} \label{section cohomo lift} Let $X = \GL_N(\R)^+/\SO(N)$ be the set of positive definite symmetric matrices, and $S = \SL_N(\R)/\SO(N)$ be the set of matrices of determinant one. For a torsion-free congruence subgroup $\Gamma \subset \SL_N(\Q)$, we consider the locally symmetric spaces 
\begin{align}
S_\Gamma=\Gamma \backslash \SL_N(\R)/\SO(N)
\end{align} and $X_\Gamma=\Gamma \backslash X$. Let $V=\Q^N \times \Q^N$ be the quadratic space of signature $(N,N)$ with quadratic form $Q(\vbf)=\langle v,w \rangle$ for a vector $\vbf=(v,w)$ in $V$. In \cite{rbrsln}, we used the Mathai-Quillen formalism \cite{MQ86} to construct a closed and $\GL_N(\R)$-invariant $N$-form
\begin{align}
\varphi(z,\vbf) \in [\Omega^{N}(X) \otimes \Scal(V_\R)]^{\GL_N(\R)},
\end{align}
valued in the space of Schwartz functions on $V_\R = \R^N \times \R^N$. By taking a regularized integral (with a complex parameter $s \in \C$) of a theta series we obtain a form
\begin{align} \label{E varphi}
E_\varphi(z_1,\tau,\phi,s) \in \Omega^{N-1}(S)^\Gamma \otimes C^\8(\HH) , \qquad z_1 \in S, \ \tau \in \HH,
\end{align}
where $\phi$ is a $\Gamma$-invariant finite Schwartz function. Moreover, in the variable $\tau$ it transforms like a modular form of weight $N$ for a congruence subgroup $\Gamma'\subseteq \SL_2(\Z)$.

Let $\Mtil_{N}(\Gamma')$ be the space of smooth functions on $\HH$ that transform like a modular form of weight $N$ for $\Gamma'$ and have moderate growth at all cusps. By integrating the above differential form above we obtain a lift 
\begin{align}
E_\varphi \colon C_{N-1}(S_\Gamma;\Z) & \longrightarrow \Mtil_{N}(\Gamma'), \\
 c & \longmapsto E_\varphi(c,\tau,\phi,s) \coloneqq \int_c E_\varphi(z_1,\tau,\phi,s).
\end{align}
from the group of smooth chains $C_{N-1}(S_\Gamma;\Z) $ to modular forms. The form $E_\varphi(c,\tau,\phi,s)$ is closed and holomorphic at $s=0$, so that it specializes to a holomorphic lift at the level of cohomology
\begin{align} \label{cohom lift}
E_\varphi \colon H_{N-1}(S_\Gamma;\Z) & \longrightarrow M_{N}(\Gamma'),\\
 [c] & \longmapsto E_\varphi(c,\tau,\phi) \coloneqq \int_c E_\varphi(z_1,\tau,\phi).
\end{align}
This lift is closely related to the lift in \cite{bcg,bcgcrm}. Moreover, if $\omega_c$ denotes a Poincaré dual to $c$, then it was shown in \cite{rbrsln} that it has the Fourier expansion
\begin{align}
E_\varphi(c,\tau,\phi)=\int_c a_0 + \sum_{n=1}^\8 \left (\int_{S_n(\phi)}\omega_c \right )e(n \tau), \qquad e(x)\coloneqq e^{2i \pi x},
\end{align}
where $a_0$ is a transgression of an Euler form as in \cite{bcg},
\begin{align} \label{special cycles}
S_n(\phi)=\sum_{\substack{[\vbf] \in \Gamma \backslash V \\ Q(\vbf)=n}} \phi(\vbf)S_{[\vbf]} \in Z^{\BM}_{\frac{N^2-N}{2}}(S_\Gamma, \Z)
\end{align}
is a locally finite cycle of codimension $N-1$ in $S_\Gamma$, and $S_{[\vbf]}$ is a generalized modular symbol obtained from an embedding of $\GL_{N-1}(\R)$ in $\SL_N(\R)$. Hence, we can interpret the $n$-th Fourier coefficient as the intersection number
\begin{align}
\int_{S_n(\phi)}\omega_c=\langle S_n(\phi), c \rangle
\end{align} between the cycle $c$ and the modular symbol, so that this lift can be seen as a Kudla-Millson lift on the symmetric space of $\SL_N(\R)$. Compare with \cite{km86,km87,KMIHES} for the lift defined by Kudla and Millson in the orthogonal setting.

\begin{ex}When $N=2$ and $\phi$ is a suitable finite Schwartz function, the space is $S_\Gamma=Y_0(p)$ and $S_n(\phi)=T_n\{0,\8\}$ is a Hecke translate of the geodesic from $0$ to $\8$.
\end{ex}

The lift \eqref{cohom lift} vanishes if the class $[c]$ is a torsion class, but it can also vanish because of the existence of a functional equation -- see Section \ref{section func}.  When it vanishes, one can instead consider its derivative
\begin{align} \label{derivative intro}
E'_\varphi \colon C_{N-1}(S_\Gamma;\Z) & \longrightarrow \Mtil_{N}(\Gamma') \\
 c & \longmapsto E'_\varphi(c,\tau,\phi) \coloneqq \dds \int_c E_\varphi(z_1,\tau,\phi,s).
\end{align}

We will define a regularized theta lift that is adjoint to $E'_\varphi$ of this lift, and use it to express the Fourier coefficients of its holomorphic projection.

%

\subsection{A regularized lift}

The Mathai-Quillen formalism produces a Thom form, that we used to define the form $\varphi(z,\vbf)$. It also defines a transgression form, that we use to define another differential form 
\begin{align}
\psi(z,\vbf) \in \Omega^{N}(X) \otimes \Scal(V_\R).
\end{align}
By replacing $\varphi$ by $\psi$ in the construction  of \eqref{E varphi}, we obtain another kernel form
\begin{align}
E_\psi(z_1,\tau,\phi) \in \Omega^{N-1}(S)^\Gamma \otimes C^\8(\HH),
\end{align}
that transforms like a modular form of weight $N-2$.

Let $H_{2-N}(\Gamma')$ be the space of weak harmonic Maass forms of weight $2-N$ for the congruence subgroup $\Gamma' \subseteq \SL_2(\Z)$. Following the seminal work of Borcherds \cite{borcherds1,borcherds2} and Bruinier-Funke \cite{BF04}, we use $E_\psi$ to define a regularized lift of $f \in H_{2-N}(\Gamma')$ by defining
\begin{align} \label{expression one}
\Phi(z_1,f,\phi) \coloneqq \CT \left \{ \lim_{T\rightarrow \8} \int_{F_T} E_{\psi}(z_1,\tau,\phi)f(\tau) y^{-\rho} \dxdy \right \},
\end{align}
where $F_T$ is a truncated fundamental domain in $\HH$ for $\Gamma'$. For each cusp $r \in \Gamma' \backslash \PP^1(\Q)$, there is a matrix $\gamma_r \in \SL_2(\Z)$ such that $\gamma_r \8=r$. Let $\phi_r \in \Scal(V_{\A_f})$ be the finite Schwartz function $\phi_r \coloneqq \omega_f(\gamma_r^{-1}) \phi$, where $\omega_f$ is the finite Weil representation. Let $a^+_r(f,n)$ be the Fourier coefficients of the holomorphic part 
\begin{align}
(f \vert_{\gamma_r})^+= \sum_{\substack{n \in \Q \\ n \gg -\8}} a^+_r(f,n)e(n\tau)
\end{align}
of $f \vert_{\gamma_r}$. Define the locus
\begin{align} 
S_f(\phi)& \coloneqq \bigcup_{r \in \Gamma' \backslash \PP^1(\Q)} \bigcup_{\substack{ n \in \Q_{>0} \\ a^+_r(f,-n) \neq 0}} S_n(\phi_r) \subset S_\Gamma,
\end{align}
where $S_n(\phi)$ are the special cycles in \eqref{special cycles}.

\begin{thm} \label{maintheorem1} The right-hand side of \eqref{expression one} converges to a differential form 
\begin{align}
\Phi(z_1,f,\phi) \in \Omega^{N-1}(S_\Gamma \smallsetminus S_f(\phi)),
\end{align}
smooth outside of $S_f(\phi)$. Moreover, it is integrable over any smooth chain $c\in C_{N-1}(S_\Gamma)$ that intersects $S_f(\phi)$ transversally.
\end{thm}
\begin{rmk}
When $N=2$ and $S=Y_0(p)$, then $S_f(\phi)$ is a collection of infinite geodesics. It would be interesting to understand if there is a relation to the locally harmonic Maass forms that appear in \cite{LS22}.
\end{rmk}
 
 \noindent We want to relate this regularized lift to \eqref{derivative intro}, the derivative of the previous theta lift. At the cusp $r \in \Gamma' \backslash \PP^1(\Q)$, the derivative $E'_\varphi(c,\tau,\phi)$ has a Fourier expansion
 \begin{align}
E'_\varphi(c,\tau,\phi)\vert_{\gamma_r}=A(c,\phi_r)+\log(y)B(c,\phi_r)+\sum_{n \in \Q} C_n(c,\phi_r,y)  e(n\tau).
\end{align}

\begin{thm}\label{main formula 1 intro}Let $f \in H_{2-N}(\Gamma')$ be a harmonic weak Maass form and $g=\xi_{2-N}(f) \in S_N(\Gamma')$ its image under the $\xi$-operator. For any smooth cycle $c$ transverse to $S_f(\phi)$ we have 
\begin{align}
2N (-1)^{N-1}\int_{c}\Phi(z_1,f,\phi) =  \langle E'_\varphi(c,\phi),g \rangle+ \sum_{ r \in \Gamma' \backslash \PP^1(\Q)}w_r \kappa_r(c,f,\phi)
\end{align}
where
\begin{align}
\kappa_r(c,f,\phi) & \coloneqq A(c,\phi_r)a^+_r(f,0)  + \lim_{T \rightarrow \8}\sum_{n \in \Q_{>0}}C_n(c,\phi_r,T)  a^+_r(f,-n).
\end{align}
\end{thm}
\noindent A form $F$ in $\Mtil_{N}(\Gamma')$ transforms like a modular form and is of moderate growth at each cusp, so that
\begin{align}
\langle F,g \rangle =\int_{\Gamma' \backslash \HH} F(\tau)\overline{g(\tau)}y^{N}\dxdy
\end{align}
converges for any cusp form $g \in S_N(\Gamma')$. In particular, the pairing $ \langle E'_\varphi(c,\phi),g \rangle$ in the previous theorem converges. The pairing defines a functional in $S_N(\Gamma')^\vee$, which is indentified with $S_N(\Gamma')$ via the Petersson inner product between holomorphic modular forms. Hence, there exists a unique cusp form $F^\hol \in S_N(\Gamma')$ such that
\begin{align}
\langle F,g \rangle=\langle F^\hol,g \rangle,
\end{align}
where the right-handside is the Petersson product. We call $F^\hol$ the holomorphic projection of $F$. 

Let us assume that $\Gamma'=\Gamma_0(l)$. For $n \in \N$, let $P_n \in S_N(\Gamma_0(l))$ be $n$-th Poincaré series of weight $N$, that satisfies
\begin{align}
\langle P_{n},g \rangle = \frac{(N-2)!}{(4\pi n)^{N-1}}a_n(g)
\end{align}
for any cusp form $g \in S_N(\Gamma_0(l))$; see \cite[\S 1.2]{bruinierborch}. (Note that when $N=2$ then the Poincar\'e series needs to be regularized). There is a harmonic weak Maass form 
\begin{align}
Q_{-n}=q^{-n}+a^+(Q_{-n},0) +O(q) \in H_{2-N}(\Gamma_0(l)),
\end{align} 
see \cite[\S6]{bor} and \cite[\S 1.3]{bruinierborch}, such that
\begin{align}
\xi_{2-N}(Q_{-n})=\frac{(4\pi n)^{N-1}}{(N-2)!}P_n.
\end{align}
In particular, we have $\langle \xi_{2-N}(Q_{-n}),g \rangle=a_n(g)$, and we deduce the following.

\begin{cor} \label{cor holo into} The holomorphic projection of $E'_\varphi(c,\phi)$ has the Fourier expansion
\begin{align}
 E'_\varphi(c,\phi)^\hol= \sum_{n =1}^\8 a_n(c,\phi)e(n \tau) \in S_N(\Gamma_0(l)),
\end{align}
where
\begin{align}
a_n(c,\phi) = 2N (-1)^{N-1}\int_{c}\Phi(z_1,Q_{-n},\phi) - \sum_{ r \in \Gamma_0(l) \backslash \PP^1(\Q)}w_r \kappa_r(c,Q_{-n},\phi).
\end{align}
\end{cor}

\subsection{Periods over tori}
In the last section of the paper, we consider the case when $c$ is a cycle obtained by embedding a totally real field $F$.  Let $D_F$ be the discriminant and $\Ocal$ the ring of integers . Assume that $F$ is of even degree $N$ and let\footnote{The existence of a totally odd character implies that the field does not contain a unit of negative norm. In particular, it excludes all fields of odd degree $N$, since $-1$ is a unit of norm $(-1)^N$} $\chi \colon \Cl(F)^+ \longrightarrow \{ \pm 1\}$ be a totally odd Hecke character on the narrow class group of $F$. For an ideal $\cfrak$ of $\Ocal$ dividing a prime $p$, consider the non-holomorphic (holomorphic at $s=0$) weight $N$ modular form for $\Gamma'=\Gamma_0(p)$
\begin{align}
\Ecal_{\chi,\cfrak}(\tau,s) \coloneqq C(s) \sum_{[\afrak] \in \Cl(F)^+}\chi(\afrak) N(\afrak)^{1+2s} \sum_{(m,n) \in \afrak\cfrak \times \afrak /\Ocal^{\times,+}} \frac{y^{Ns}}{\N(v\tau+w) \vert \N(v\tau+w) \vert^{2s}},
\end{align}
where $C(s)$ is a term involving the discriminant of $F$, powers of $\pi$ and of the Gamma function. With this normalization, it satisfies the functional equation
\begin{align}
\Ecal_{\chi,\cfrak}(\tau,-s)=\chi(\dfrak\cfrak)N(\dfrak\cfrak)^{2s}\Ecal_{\chi,\cfrak}(\tau,s).
\end{align}
In particular, if $\chi(\cfrak)=-\chi(\dfrak)$, then the Eisenstein series $\Ecal_{\chi,\cfrak}(\tau,s)$ vanishes at $s=0$. Instead, we consider its derivative
\begin{align}
\Ecal'_{\chi,\cfrak}(\tau) \coloneqq \dds \Ecal_{\chi,\cfrak}(\tau,s).
\end{align}
It is a non-holomorphic modular form of weight $N$ for $\Gamma_0(p)$ and has a Fourier expansion
\begin{align}
\Ecal'_{\chi,\cfrak}(\tau)= \log(\alpha(\chi,\cfrak))+A_{\chi}+B_{\chi}\log(y)+\sum_{n=1}^\8\log(J_{\chi,\cfrak}(n))e(n \tau) + \sum_{n \in \Z} a_{\chi,\cfrak}(n,y)e(n \tau)
\end{align}
where $\alpha(\chi,\cfrak), J_{\chi,\cfrak}(n)$ and $B_{\chi}$ are algebraic numbers, and $A_{\chi}$ is transcendental. Let $\epsilon$ be a basis of $F$ over $\Q$, and 
\begin{align}
c_\epsilon \in Z_{N-1}(S_\Gamma)
\end{align} be a cycle obtained from the embedding of the torus $F^\times$ in $\GL_N(\Q)$ induced by $\epsilon$. For a suitable finite Schwartz function $\phi_{\chi,\cfrak}$, the Eisenstein series can be obtained as the image of the cycle $c_\epsilon$ by the lift \eqref{cohom lift}, so that we have
\begin{align}
\Ecal_{\chi,\cfrak}(\tau,s)=E_\varphi(c_\epsilon,\tau,\phi_{\chi,\cfrak},s).
\end{align}

\begin{thm} \label{main thm intro} Let $f \in H_{2-N}(\Gamma_0(p))$ be a harmonic weak Maass form and $g=\xi_{2-N}(f) \in S_N(\Gamma_0(p))$. Suppose that the principal parts of $f$ at all cusps have rational coefficients. Then
\begin{align}
\int_{c_\epsilon}\Phi(z_1,f,\phi_{\chi,\cfrak}) = \log(\alpha(f,\chi,\cfrak))-\frac{1}{2N} \langle \Ecal'_{\chi,\cfrak},g \rangle -\frac{A_{\chi}}{2N} \left (a^+_\8(f,0)+p \frac{\chi(\cfrak)}{\N(\cfrak)} a^+_0(f,0) \right ),
\end{align}
for an algebraic number $\alpha(f,\chi,\cfrak) \in \Qbar$.
\end{thm}
\noindent The algebraic number $\alpha(f,\chi,\cfrak)$ is a $d$-th root of a rational number, so if we replace $f$ by $df$, then $\alpha(df,\chi,\cfrak)=\alpha(f,\chi,\cfrak)^d \in \Q$. We use the fact that the $L$-function $L(\chi,0)$ is a rational number, by Shintani's Theorem.

%

\subsection{Overview} In Section \ref{section on theta lift}, we recall the construction of the theta lift on $S$, using differential forms obtained from the Mathai-Quillen formalism, see \ref{Mathai-Quillen}. The special cycles $S_{[\vbf]}$ are defined in \ref{section special cycles}. In  \ref{definition of psi}, we define the form $\psi$, that is used to define the kernel $E_\psi(z_1,\tau,\phi)$, and the regularized lift $\Phi(z_1,f,\phi)$ in \ref{definition of regularized lift}. The integrability of $\Phi(z_1,f,\phi)$ is proved in Theorem \ref{proplim2}. The relation to the derivative of $E_{\varphi}(z_1,\tau,\phi)$ is proved in Theorem \ref{main formula 1}, and follows from the fact that $L_N E'_{\varphi}(z_1,\tau,\phi)+2N(-1)^{N-1}E_{\psi}(z_1,\tau,\phi)$ is an exact form by Proposition \ref{lemmaderiv}, where $L_N$ is the lowering operator.

\section{Theta lift on $\SL_N$} \label{section on theta lift}
We recall the construction of theta lift from \cite{rbrsln}. Let $W$ be an $N$-dimensional $\Q$-vector space. Let $X$ be the space of positive-definite quadratic forms on $W_\R$. After fixing a basis $W \simeq \Q^N$, we can identify $X$ with the space of positive-definite symmetric $N \times N$ matrices. The group $G\coloneqq \GL_N(\R)^+$ acts on $X$ by $z \longmapsto gzg^t$, and the stabilizer of $\id_{N\times N}$ is $K\coloneqq \SO(N)$. Thus, we have
\begin{align}
X \simeq G/K.
\end{align}
It is a smooth manifold of dimension $\frac{N^2+N}{2}$. Let $S \subset X$ be the submanifold of matrices of determinant one, that can be described as the homogeneous space
\begin{align}
S \simeq G^1/K,
\end{align}
where $G^1 \coloneqq \SL_N(\R)$. We have a  diffeomorphism
\begin{align}
S \times \R_{>0} \longrightarrow X, \qquad  (z_1,u) \longmapsto z=u^2z_1,
\end{align}
where $u=\det(z)^{\frac{1}{2N}}$. We view $X$ as a bundle of rank one over $S$, with the projection map
\begin{align}
\pi \colon X \longrightarrow S, \quad z_1=\pi(z)\coloneqq u^{-2}z.
\end{align}
Let $V \coloneqq W \times W^\vee$ be the quadratic space with the quadratic form of signature $(N,N)$
\begin{align}
Q(\vbf,\vbf') \coloneqq \langle v,w' \rangle+\langle v',w \rangle \qquad \vbf=(v,w), \ \vbf'=(v',w'),
\end{align}where $\langle v,w \rangle \coloneqq v^tw$. We set
\begin{align}
Q(\vbf) \coloneqq  \frac{1}{2}Q(\vbf,\vbf)=\langle v,w \rangle.
\end{align} 
The group $G$ acts on a vector $\vbf=(v,w)$ in $V$ by
\begin{align}
\rho(g)(\vbf)=(gv,g^{\vee}w),
\end{align}
where $g^{\vee}$ is the induced action on the dual space. This gives a representation $\rho \colon G \longrightarrow \SO(V_\R) =\SO(N,N)$. From now on we will identify $W \simeq \Q^N$ and $W^\vee \simeq \Q^N$ with the standard  inner product $\langle v,w \rangle = v^tw$, so that $g^{\vee}=g^{-t}$.

\subsection{Special cycles in the locally symmetric space} \label{section special cycles} For a place $v$ of $\Q$ let $W_{\Q_v} \coloneqq W \otimes \Q_v$. Similarly, let $V_{\Q_v}=V \otimes \Q_v$ be the quadratic space with the quadratic form $Q$ extended to $\Q_v$. We have fixed a basis of $W$ so that we identify $V_{\Q_v} \simeq \Q_v^{2N}$, and for a finite place $v=p$ let $V_{\Z_p}=\Z_p^{2N}$. Let $V_{\A_f}=\sideset{}{^{'}} \prod_v V_{\Q_v}$ be the finite adeles, where the product is the restricted product with respect to $V_{\Z_p}$.

Let $\Scal(V_{\A_f})$ be the space of finite Schwartz function. The group $\SL_N(\Q)$ acts on finite Schwartz function $\phi \in \Scal(V_{\A_f})$ by $\rho$
\begin{align} \label{invariance of phi}
(\rho(g)\phi)(v,w) = \phi(gv,g^{-t}w).
\end{align}
Let us now fix a finite Schwartz function $\phi \in \Scal(V_{\A_f})$ and $\Gamma \subset \SL_N(\Q)$ be a torsion-free subgroup preserving $\phi$ under $\rho$. We consider the locally symmetric spaces
\begin{align}
S_\Gamma \coloneqq \Gamma \backslash S \qquad \textrm{and} \qquad X_\Gamma \coloneqq \Gamma \backslash X \simeq S_\Gamma \times \R_{>0}.
\end{align}
For each vector $\vbf=(v,w) \in V$ we define the submanifold
\begin{align}
X_\vbf \coloneqq \left \{ z \in X \ \vert zw=v \right \} \subseteq X.
\end{align}
Note that if $z \in X_\vbf$ for some vector $\vbf \in V$, then $v^tzw=v^tv \geq 0$. A vector $\vbf=(v,w)$ is {\em regular} if $v,w$ are both nonzero, and {\em singular} otherwise. Thus, if $\vbf$ is a regular vector with $Q(\vbf)>0$, then $X_\vbf$ is a submanifold of codimension $N$ in $X$. If $\vbf=0$, then $X_\vbf=X$, and otherwise $X_\vbf$ is empty. Let \begin{align}
S_\vbf \coloneqq \pi(X_\vbf) \subseteq S,
\end{align}
be the image under the projection $\pi \colon X \longrightarrow S$. If $\vbf$ is a regular vector with $Q(\vbf)>0$, then the restriction of the projection $\pi \colon X \longrightarrow S$ to $X_\vbf$ is a diffeomorphism onto $S_\vbf$, which is of codimension $N-1$ in $S$. The submanifolds $X_\vbf$ and $S_\vbf$ satisfy the invariance property
\begin{align}
gX_\vbf=X_{\rho(g) \vbf} \qquad \textrm{and} \qquad \ g_1S_\vbf=S_{\rho(g_1) \vbf},
\end{align}
where we recall that $g$ acts on $z$ by $gzg^t$.

\begin{rmk} The representation $\rho \colon G \longrightarrow \SO(V_\R) =\SO(N,N)$ induces an embedding of $X$ as a smooth submanifold of $\D \simeq \SO(N,N)/K\times K$, the Grassmannian of negative $N$-planes that appears in the work of Kudla-Millson.
In particular, we have $X_\vbf =X \cap \D_\vbf$ where $\D_\vbf$ are the planes in $\D$ that are orthogonal to $\vbf$ with respect to $Q$.
\end{rmk}
 Define the group of locally finite cycles 
\begin{align}
Z^{\BM}_{k}(S_\Gamma, \Z) \coloneqq Z_{k}(\overline{S}_\Gamma, \{\8\}, \Z),
\end{align}
where $\overline{S}_\Gamma \coloneqq S_\Gamma \cup \{\8\}$ is the one-point compactification of $S_\Gamma$. The homology of this complex is the Borel-Moore homology.
 Let us denote by $S_{[\vbf]}$ the image of $S_\vbf$ under the projection map $S \longrightarrow S_\Gamma$. By the equivariance, it only depends on the class $[\vbf]$ of $\vbf$ in $\Gamma \backslash L$. It defines a locally finite cycle
\begin{align}
S_{[\vbf]} \in Z^{\BM}_{\frac{N^2-N}{2}}(S_\Gamma; \{\8\}, \Z)
\end{align}
of codimension $N-1$ in $S_\Gamma$. For a positive integer $n$, we define the special cycle
\begin{align}
S_n(\phi) \coloneqq \sum_{\substack{ \vbf \in \Gamma \backslash V \\ Q(\vbf)=n}} \phi(\vbf)S_{[\vbf]} \in  Z^{\BM}_{\frac{N^2-N}{2}}(S_\Gamma, \Z),
\end{align}
where the sum is finite.

\subsection{Differential forms} \label{Mathai-Quillen} Over $X=G/K$ we have a bundle
\begin{align}
E=G \times_{K} V_\R
\end{align}
that consists equivalence classes of pairs $[g,v] \in G \times V_\R $,  modulo the equivalence relation $[g,v] \equiv [gk,k^{-1}v]$. We can view it as a metric bundle by taking the constant metric $\lVert [g,v] \rVert^2 \coloneqq \lVert v \rVert^2$ in each fiber.  This bundle is isometric to the tautological metric bundle
\begin{align}
X \times \R^N, \qquad \lVert (z,v) \rVert^2 \coloneqq v^tz^{-1}v,
\end{align}
where the metric over a point $z \in X$ is given by the positive-definite quadratic form corresponding to $z$.  The isometry is given by the map 
\begin{align}
E \longrightarrow X \times V_\R, \quad  [g,v] \longmapsto (z,gv),
\end{align}
where $z=gK$. The group $G$ acts on $E$ by $g \cdot [g',v]=[gg',v]$, and is $G$-equivariant. In particular it is $\Gamma_\vbf$-equivariant, where $\Gamma_\vbf$ is the stabilizer of $\vbf$ under the action by $\rho$. The quotient $E_{\Gamma_\vbf} \coloneqq {\Gamma_\vbf}  \backslash E$ is a bundle over $X_{\Gamma_\vbf}$. After fixing an orientation of the fibers, the pushforward along the fiber induces the Thom isomorphism
\begin{align}
H^{k+N}(E_{\Gamma_\vbf} ,E_{\Gamma_\vbf}  \smallsetminus E_0;\R) \longrightarrow H^k(X_{\Gamma_\vbf} ;\R),
\end{align}
where $E_0$ is the image of the zero section. Taking $k=0$, a canonical class is given by taking the preimage of a generator of $H^0(X_{\Gamma_\vbf} ;\R) \simeq \R$. It is called the {\em Thom class} of the vector bundle. Mathai and Quillen constructed in \cite{MQ86} an explicit representative
\begin{align}
U \in \Omega^{N}(E)^{G} 
\end{align}
of this class, which is closed, rapidly decreasing along the fiber and $G$-invariant. In particular, it is $\Gamma_\vbf$-invariant and represents the Thom class on $E_{\Gamma_\vbf}$. If $s_\vbf \colon X_{\Gamma_\vbf}  \longrightarrow E_{\Gamma_\vbf} $ is a section that is transverse to the zero section $E_0$, then $s_\vbf^{-1}(E_0)$ is a submanifold of codimension $N$, and $s_\vbf^\ast U$ is a Poincaré dual to $s_\vbf^{-1}(E_0)$.

\subsubsection{Pullback of the Thom form}The submanifold $X_\vbf$ is precisely the zero locus of the section
\begin{align} \label{sectionsv}
s_\vbf \colon X \longrightarrow E, \quad z \longmapsto \left [ g,\frac{g^{-1}v-g^t w}{2}\right ],
\end{align}
where $\vbf=(v,w)$.The pullback by the section $s_\vbf$ gives a closed form
\begin{align}
\varphi^0(z,\vbf)\coloneqq s_\vbf^\ast U \in \Omega^N(X) \otimes C^\8(V_\R)
\end{align} 
which satisfies the invariance property
\begin{align}
g^\ast \varphi^0(z,\vbf)=\varphi^0(z,\rho_{g^{-1}}\vbf).
\end{align}
\begin{rmk}
Note that section is $\Gamma_\vbf$-equivariant, and we can view it as a section $s_\vbf \colon X_{\Gamma_\vbf}  \longrightarrow E_{\Gamma_\vbf}$. Then $\varphi^0(z,\vbf)$ represents the Poincaré dual in $H^{N}(X_{\Gamma_\vbf},\R)$, of the image of $\Gamma_\vbf \backslash X_{\vbf} \hooklongrightarrow \Gamma_\vbf \backslash X$.
\end{rmk}

The construction of Mathai-Quillen is explicit and we can compute this form as follows. The Maurer-Cartan form is the $1$-form
\begin{align}
\vartheta\coloneqq g^{-1}dg \in \Omega^1(G)\otimes \g,
\end{align}
where $\g \simeq \Mat_N(\R)$ is the Lie algebra of $G$. Let $\lambda_{ij} \in \Omega^1(G)$ be the $(i,j)$-entry of 
\begin{align} \label{form lambda}
\lambda \coloneqq \frac{1}{2}(\vartheta+\vartheta^t) \in \Omega^1(G) \otimes \g.
\end{align} For a subset $I=\{i_1<\dots<i_k\} \subseteq \{1,\dots,N\}$ and a function $\sigma \colon I \longrightarrow \{ 1, \dots, N \}$ we define the $k$-form
\begin{align}
    \lambda(\sigma) \coloneqq \lambda_{i_1\sigma(1)} \wedge  \cdots \wedge \lambda_{i_k\sigma(N)} \in \Omega^{k}(G).
\end{align}
For a tuple $\uud=(d_1, \dots,d_N)$ such that $\vert \uud \vert \coloneqq d_1+\cdots+d_N=\vert I \vert$ we set
\begin{align}
\lambda_I(\uud) \coloneqq \sum_{\substack{\sigma \\ \vert \sigma^{-1}(m) \vert=d_m}} \lambda(\sigma),
\end{align}
where the sum is over all functions $\sigma \colon I \longrightarrow \{ 1, \dots, N \}$. When $I=\{ 1, \dots, N \}$ we write
\begin{align}
\lambda(\uud) \coloneqq \lambda_{\{ 1, \dots, N \}}(\uud).
\end{align}
We also define a polynomial $H_{\uud} \in \C[\R^{N}]$ by
\begin{align}
    H_{\uud}(v)  \coloneqq \prod_{m=1}^NH_{d_m}(\langle v,e_m\rangle),
\end{align}
where $H_d$ denotes the single variable Hermite polynomial
\begin{align}
H_d(t) \coloneqq \left (2x-\frac{d}{dt} \right )^d \cdot 1.
\end{align} The first few Hermite polynomials are $H_1(t)=2t$, $H_2(t)=4t^2-2$ and $H_3(t)=8t^3-12t$. One can compute\footnote{See \cite[Proposition.~3.7]{rbrsln}. Note that there is a mistake at several places of {\em loc. cit.}, where $H_d(-\sqrt{\pi}(v+w))$ should have been $H_d(\sqrt{\pi}(v+w))$.} that
\begin{align} \label{formula varphi zero}
\varphi^0(z,\vbf)=\frac{1}{2^{N}\pi^{\frac{N}{2}}}e^{-\pi \lVert g^{-1}v-g^t w \rVert^2} \sum_{\vert \uud \vert=N } H_{\uud}(\sqrt{\pi}(g^{-1}v+g^t w)) \lambda(\uud).
\end{align}
Note that here we view the right-hand side as a $K$-invariant (basic) form on $G$ that descends to $X$.

The form $\varphi^0$ is smooth in $V_\R$ but not rapidly decreasing in $V_\R$. We obtain a rapidly decreasing form by setting
\begin{align}
\varphi(z,\vbf) \coloneqq e^{-2\pi Q(\vbf)} \varphi^0(z,\vbf) \in \Omega^N(X) \otimes \Scal(V_\R).
\end{align}
This form is also closed and satisfies
\begin{align} \label{equivariance}
g^\ast \varphi(z,\vbf)=\varphi(z,\rho_{g^{-1}}\vbf).
\end{align}
We have 
\begin{align} \label{expression varphi}
\varphi(z,\vbf)=\frac{1}{2^{N}\pi^{\frac{N}{2}}}e^{-\pi \lVert g^{-1}v \rVert^2-\pi \lVert g^t w \rVert^2} \sum_{\vert \uud \vert=N } H_{\uud}(\sqrt{\pi}(g^{-1}v+g^t w)) \lambda(\uud).
\end{align}

\subsubsection{The transgression form} The Thom form of Mathai-Quillen comes with a {\em transgression form} on $E$ of degree $N-1$. By pulling back by $s_\vbf$ it gives rise to a form 
\begin{align} \alpha^0(z,\vbf) \in \Omega^{N-1}(X) \otimes C^\8(V_\R),
\end{align}
that satisfies the {\em transgression formula}
\begin{align}
d_X\alpha^0(z,\sqrt{y}\vbf)= y \frac{\partial}{\partial y}\varphi^0(z,\sqrt{y}\vbf);
\end{align}
see \cite[Proposition.~3.12]{rbrsln}.
Similarly to the form $\varphi^0$, the form is given by an explicit formula
\begin{align} \label{formalpha}
\alpha^0(z,\vbf) =-\frac{1}{2^{N}\pi^{\frac{N-1}{2}}} e^{ -\pi \lVert g^{-1}v - g^t w \rVert^2}\sum_{l=1}^N (-1)^{l-1} \langle g^{-1}v-g^t w, e_l \rangle  \sum_{\vert \uud \vert=N-1} H_{\uud}(\sqrt{\pi}(g^{-1}v+g^t w)) \lambda_{I_l}(\uud) \qquad
\end{align}
where $I_l \coloneqq \{1, \dots,N\} \smallsetminus \{l\}$. The form $\alpha^0$ is also not rapidly decreasing in $V_\R$ and we define
\begin{align}
\alpha(z,\vbf) \coloneqq \alpha^0(z,\vbf)\exp(-2\pi Q(\vbf)) \in \Omega^N(X) \otimes \Scal(V_\R).
\end{align}
We have 
\begin{align}
\alpha(z,\vbf) =-\frac{1}{2^{N}\pi^{\frac{N-1}{2}}} e^{ -\pi \lVert g^{-1}v \rVert^2-\pi \lVert g^t w \rVert^2} \sum_{l=1}^N (-1)^{l-1}\langle g^{-1}v-g^t w, e_l \rangle \sum_{\vert \uud \vert=N-1} H_{\uud}(\sqrt{\pi}(g^{-1}v+g^t w)) \lambda_{I_l}(\uud).
\end{align}

\subsection{Weil representation} The next step is to construct theta series using the previously constructed $\Scal(V_\R)$-valued differential forms and the Weil representation. We consider the Weil representation in two different models. Let $\Scal(V_{\Q_v})$ be the space of Schwartz-Bruhat functions. In the first model, the representation $\omega \colon \GL_N(\Q_v) \times \SL_2(\Q_v) \longrightarrow U(\Scal(V_{\Q_v}))$ is defined by the formulas
\begin{align*}
\omega_v\left (g, h \right ) \varphi_v(\vbf)&=\omega_v(h)\varphi_v(\rho(g)^{-1}\vbf), \\
    \omega\left (1, \begin{pmatrix}
        a & 0 \\ 0 & a^{-1}
    \end{pmatrix} \right ) \varphi_v(\vbf)&=\vert a \vert_v^N \varphi_v(a\vbf), \qquad a \in \Q_v^\times, \\
    \omega_v\left (1, \begin{pmatrix}
        1 & n \\ 0 & 1
    \end{pmatrix} \right ) \varphi_v(\vbf)&=e \left ( nQ(\vbf) \right ) \varphi_v(\vbf), \qquad n \in \Q_v, \\
    \omega_v\left (1, \begin{pmatrix}
        0 & -1 \\ 1 & 0
    \end{pmatrix} \right ) \varphi_v(\vbf)&= \int_{V_{\Q_v}} \varphi_v(\vbf') e_v \left (- Q(\vbf,\vbf') \right )d\vbf',
\end{align*}
where $e_\8(t)=e^{2i\pi t}$ if $v=\8$, and $e_p(t)=e^{-2i\pi \{t\}_p}$ if $v=p$. A second model $\omega' \colon \GL_N(V_{\Q_v}) \times \SL_2(\Q_v) \longrightarrow U(\Scal(V_{\Q_v}))$ for the Weil representation is given by
\begin{align}
    \omega_v'(g,h)\varphi_v(\vbf)\coloneqq \vert \det(g)\vert_v^{-1}\varphi_v(g^{-1}\vbf h)
\end{align}
where $h$ acts on the two columns of $\vbf=(v,w)$ from the right. The two representations are intertwined by partial Fourier transform in the second variable
\begin{align}
    \Scal(V_{\Q_v}) \longrightarrow \Scal(V_{\Q_v}), \qquad \varphi_v \longmapsto \widehat{\varphi_v}(v,w) \coloneqq \int_{V_{\Q_v}} \varphi_v(v,w') e(\langle w,w' \rangle)dw'.
\end{align}

These representations extend to representations of the space of finite Schwartz function $\Scal(V_{\A_f})$. We will denote by $\omega_f,\omega'_f$ the Weil representations
\begin{align}
\omega_f ,\omega'_f \colon \GL_N(V_{\A_f}) \times \SL_2(\A_f) \longrightarrow U(\Scal(V_{\A_f})).
\end{align} Similarly, the partial Fourier transforms extends to a map $\phi \in \Scal(V_{\A_f}) \longmapsto \vphat \in \Scal(V_{\A_f})$.  
%
By Poisson summation, for $\varphi \in \Scal(V_\R)$ and $\phi \in \Scal(V_{\A_f})$, we have
\begin{align}
\sum_{\vbf \in V} (\omega_\8(g,h)\varphi)(\vbf)\phi(\vbf)=\sum_{\vbf \in V} (\omega_\8'(g,h)\widehat{\varphi})(\vbf)\vphat(\vbf).
\end{align}
\begin{prop}\label{prop fourier transform}\cite[Proposition.~4.9]{rbrsln} For a vector $v \in \C^N$ and an $N$-tuple $\uud$ let us write $v^{\uud} \coloneqq \langle v,e_1\rangle^{d_1} \cdots\langle v,e_N\rangle^{d_N},$ where $e_l \in \R^N$ is the $l$-th basis vector.
The partial Fourier transforms of $\varphi(z,\vbf)$ and $\alpha(z,\vbf)$ are
 \begin{align} \label{partial fourier vphat}
\widehat{\varphi}(z,\vbf)=\frac{i^N}{\det(g)} e^{-\pi  \lVert g^{-1}v\rVert^2-\pi \lVert g^{-1} w \rVert^2} \sum_{\vert \uud \vert=N}  \overline{g^{-1}(vi+w)}^{\uud}  \lambda(\uud)
\end{align}
 and
 \begin{align}
\alphahat(z,\vbf) =-\frac{i^N}{\det(g)}e^{-\pi  \lVert g^{-1}v\rVert^2-\pi \lVert g^{-1} w \rVert^2} \sum_{l=1}^N (-1)^{l-1}\sum_{\vert \uud \vert=N-1} \left [ \frac{2\pi \vert \langle g^{-1}(vi+w),e_l\rangle \vert ^2+d_l}{2\pi\langle g^{-1}(\overline{vi+w}),e_l\rangle}\right ] \overline{g^{-1}(vi+w)}^{\uud}  \lambda_{I_l}(\uud),
\end{align}
where $z=gg^t$.
\end{prop}

\subsection{Example when $N=2$} \label{example lambda} We have $X \simeq \HH \times \R_{>0}$. Let $z_1=x_1+iy_1$ be the coordinate on $S \simeq \HH$. Using the formula in Proposition \ref{prop fourier transform} we find
 \begin{align}
\alphahat(z,\vbf) & =u^{-2}\exp \left (-u^{-2}\pi \lVert g_1^{-1}(vi+w) \rVert^2 \right ) \left [ \left ( \vert \langle g^{-1}(vi+w),e_1\rangle \vert ^2+\frac{1}{2\pi}\right ) \lambda_{I_1}(1,0) \right .  \\
& + \langle g^{-1}(vi+w),e_1\rangle \langle \overline{g^{-1}(vi+w)},e_2\rangle  \lambda_{I_1}(0,1) -  \langle \overline{g^{-1}(vi+w)},e_1\rangle \langle g^{-1}(vi+w),e_2\rangle   \lambda_{I_2}(1,0) \\
& \hspace{5cm}-\left . \left ( \vert \langle g^{-1}(vi+w),e_2\rangle \vert ^2+\frac{1}{2\pi}\right ) \lambda_{I_2}(0,1) \right ].
\end{align}
The form $\lambda$ is
\begin{align}
\lambda = \frac{1}{2} \left (g_1^{-1}dg_1+(g_1^{-1}dg_1)^t \right )+\frac{du}{u}=\begin{pmatrix}
\frac{du}{u} + \frac{dy_1}{2y_1}& \frac{dx_1}{2y_1} \\ \frac{dx_1}{2y_1} &\frac{du}{u} - \frac{dy_1}{2y_1} \end{pmatrix} \in \Omega^2(X) \otimes \End(\R^2)
\end{align}
where
\begin{align}
g_1=\begin{pmatrix}
\sqrt{y_1} & x_1/\sqrt{y_1} \\ 0 & 1/\sqrt{y_1}
\end{pmatrix}.
\end{align} For $I_1=\{2\}$ and $I_2=\{1\}$ we have
\begin{align}
\lambda_{I_1}(1,0) &=\lambda_{21}= \frac{dx_1}{2y_1}, \qquad \lambda_{I_2}(1,0)=\lambda_{11}=\frac{du}{u} + \frac{dy_1}{2y_1} ,\\
\lambda_{I_1}(0,1)&=\lambda_{22}=\frac{du}{u} - \frac{dy_1}{2y_1}, \qquad \lambda_{I_2}(0,1)=\lambda_{12}=\frac{dx_1}{2y_1}.
\end{align}
Hence, we can rewrite $\alphahat(z,\vbf)$ as
\begin{align}
\alphahat(z,\vbf)=\alphahat_{x_1}(g,\vbf)\frac{dx_1}{2y_1} +\alphahat_{y_1}(g,\vbf)\frac{dy_1}{2y_1}+\alphahat_{u}(g,\vbf)\frac{du}{2u},
\end{align}
where
\begin{align} \label{formulaalphaxyu}
\alphahat_{x_1}(g,\vbf) &\coloneqq u^{-4}e^{-u^{-2}\pi \lVert g_1^{-1}(vi+w) \rVert^2} \left ( \vert \langle g_1^{-1}(vi+w),e_1\rangle \vert^2-\vert \langle g_1^{-1}(vi+w),e_2\rangle\vert^2 \right ), \nonumber \\
\alphahat_{y_1}(g,\vbf)&\coloneqq-2u^{-4}e^{-u^{-2}\pi \lVert g_1^{-1}(vi+w) \rVert^2} \re \left ( \langle g_1^{-1}(\overline{vi+w}),e_1\rangle \langle g_1^{-1}(vi+w),e_2\rangle\right ), \\
\alphahat_{u}(g,\vbf)&\coloneqq 2iu^{-4}e^{-u^{-2}\pi \lVert g_1^{-1}(vi+w) \rVert^2} \im \left ( \langle g_1^{-1}(\overline{vi+w}),e_1\rangle \langle g_1^{-1}(vi+w),e_2\rangle\right ).
\end{align}

\subsection{Theta series}  From Proposition \ref{prop fourier transform} it follows that
\begin{align} 
\omega_\8'(1,k)\widehat{\varphi}(z,\vbf) & =j(k,i)^{-N}\widehat{\varphi}(z,\vbf), \\
\omega_\8'(1,k)\alphahat(z,\vbf)&=j(k,i)^{-(N-2)}\alphahat(z,\vbf)
\end{align}
for $k \in \SO(2)$, where $j(g,\tau)=c\tau+d$ for $g=\begin{pmatrix} a & b \\ c & d \end{pmatrix}$. Since the partial Fourier transform intertwines $\omega_\8$ and $\omega_\8'$, we also have
\begin{align} \label{indep}
\omega_\8(1,k)\varphi(z,\vbf)&=j(k,i)^{-N}\varphi(z,\vbf) \\
\omega_\8(1,k)\alpha(z,\vbf)&=j(k,i)^{-(N-2)}\alpha(z,\vbf).
\end{align}

 For $\tau \in \HH$ and $\phi \in \Scal(V_{\A_f})$ we set
\begin{align}
\Theta_{\varphi}(z,\tau,\phi) & \coloneqq j(h_\tau,i)^{N} \sum_{\vbf \in V} \omega_\8(1,h_\tau)\varphi(z,\vbf)\phi(\vbf) \in \Omega^N(X) \otimes C^\8(\HH),
\end{align}
where $h_\tau \in \SL_2(\R)$ is any matrix sending $i$ to $\tau$. By \eqref{indep}, the right-hand side only depends on $\tau$ and not on the choice of $h_\tau$. Taking the standard matrix
$h_\tau=\begin{pmatrix}
\sqrt{y} & x/\sqrt{y} \\ 0 & 1/\sqrt{y}
\end{pmatrix}$ we get
\begin{align}
\Theta_{\varphi}(z,\tau,\phi) & = \sum_{\vbf \in V} \varphi^0(z,\sqrt{y}\vbf) \phi(\vbf) e(Q(\vbf)\tau).
\end{align}
The differential form $\varphi$ satisfies the equivariance $g^\ast \varphi(z,\vbf)=\varphi(z,\rho_{g^{-1}}\vbf)=\omega_\8(g,1)\varphi(z,\vbf)$. Moreover, the invariance in \eqref{invariance of phi} can be written $\omega_f(\gamma,1)\phi=\phi$. Thus, if $\Gamma \subset \SL_N(\Q)$ is a torsion-free subgroup preserving $\phi$ as in \ref{section special cycles}, then
\begin{align}
\gamma^\ast \Theta_{\varphi}(z, \tau,\phi)=\Theta_{\varphi}(z,\tau,\omega_f(\gamma^{-1},1) \phi)=\Theta_{\varphi}(z,\tau, \phi) \in \Omega^{N}(X)^{\Gamma} \otimes C^\8(\HH)
\end{align}
defines a $\Gamma$-invariant form on $X$. On the other hand, for $\gamma \in \SL_2(\Q)$ it satisfies
\begin{align} \label{slash Epsi}
\Theta_{\varphi}(z,\gamma \tau,\phi)=j(\gamma,\tau)^N \Theta_{\varphi}(z,\tau,\omega_f(1,\gamma^{-1}) \phi).
\end{align}
Thus, if $\Gamma' \subseteq \SL_2(\Z)$ is a congruence subgroup that preserves $\phi$ under $\omega_f$, then if transforms like a modular form of weight $N$ for $\Gamma' \subseteq \SL_2(\Z)$.

Similarly, we define the theta series 
\begin{align} \label{thetalpha}
\Theta_{\alpha}(z,\tau,\phi)&=j(h_\tau,i)^{(N-2)} \sum_{\vbf \in V} \omega_\8(1,h_\tau)\alpha(z,\vbf)\phi(\vbf) \in \Omega^{N-1}(X)^{\Gamma} \otimes C^\8(\HH) \\
& =y\sum_{\vbf \in V} \alpha^0(z,\sqrt{y}\vbf) \phi(\vbf)e(Q(\vbf)\tau)
\end{align}
which transforms like a modular form of weight $N-2$ and level $\Gamma_0(p)$ . Moreover, we have
\begin{align} 
d_X\alpha^0(z,\sqrt{y}\vbf)=y^2 \frac{\partial}{\partial y}\varphi^0(z,\sqrt{y}\vbf)=L_N \varphi^0(z,\sqrt{y}\vbf)
\end{align}
where $L_N=-2iy^2\dbar,$ is the lowering operator and $d_X$ is the exterior derivative on $X$. Thus, we have \cite[Theorem.~4.19]{rbrsln}
\begin{align} \label{transgressionformula}
d_X\Theta_{\alpha}(z,\tau,\phi)=L_N \Theta_{\varphi}(z,\tau,\phi).
\end{align}

\subsection{Push-forward} A matrix $g \in G$ can be written $g=g_1u$ where $g_1 \in G^1$ and $u=\det(g)^{\frac{1}{N}}$. This induces an isomorphism
\begin{align}
X \simeq S \times \R_{>0}, \quad z \longmapsto (z_1,u)
\end{align} and we can view $X$ as a bundle $\pi \colon X \longrightarrow S$ of rank $1$ over $S$. A form $\eta \in \Omega^k(X)$ can be written $\eta=\eta_0+\eta_1\frac{du}{u}$ where $\eta_0$ is a $k$-form on $X$ of degree $0$ in $u$, and $\eta_1=(-1)^{N-1}\iotau \eta$ is a $(k-1)$-form of degree $0$ in $u$. The pushforward along $\pi$ is defined by
\begin{align}
\pi_\ast(\eta) \coloneqq \int_0^\8 \eta_1\frac{du}{u} \in \Omega^{k-1}(S),
\end{align}
provided that the integral converges. We define the forms
\begin{align} \label{define Ephi}
E_\varphi(z_1,\tau,\phi,s) \coloneqq \pi_\ast \left ( \Theta_{\varphi}(z,\tau,\phi)u^{-2Ns} \right )= (-1)^{N-1}\int_0^\8 \iotau \Theta_{\varphi}(z,\tau,\phi)u^{-2Ns} \frac{du}{u},
\end{align}
and  $E_\alpha(z_1,\tau,\phi,s)$ similarly. By \cite[Theorem.~4.19]{rbrsln}, we have 
\begin{align} \label{transgressionformula2}
d_SE_{\alpha}(z_1,\tau,\phi)=L_N E_{\varphi}(z_1,\tau,\phi),
\end{align}
where $d_S$ is the exterior derivative on $S$. Instead of the $(N-1)$-form $\alpha(z,\vbf)$ we will rather use the $N$-form
\begin{align} \label{definition of psi}
\psi(z,\vbf) \coloneqq \alpha(z,\vbf)\frac{du}{u}.
\end{align}
\begin{rmk} The construction of Bruinier-Funke \cite{BF04} would suggest to use the $(N-1)$-form $\alpha$, whose pushforward is an $(N-2)$-form. The motivation to use the form $\psi$ is its relation to the derivative of $E_{\varphi}(z_1,\tau,\phi,s)$.  On the other hand, if we had used the form $\alpha$, then the exterior derivative of the corresponding regularized lift would be related to $E_{\varphi}(z_1,\tau,\phi)$ by a current equation, as in \cite[Theorem.~1.5]{BF04}.  In particular, it should define a differential character in the sense of Cheeger-Simons. We do not have such an interpretation for the corresponding regularized lift $\Phi(z_1,f,\phi)$ obtained from $\psi$ in \eqref{definition of regularized lift}. In future work, we hope to modify the construction presented in this paper to combine these two point of views.
\end{rmk}

Let $\Theta_{\psi}(z,\tau,\phi)$ be the theta series defined as in \eqref{thetalpha}
\begin{align}
\Theta_{\psi}(z,\tau,\phi) \coloneqq \Theta_{\alpha}(z,\tau,\phi)\frac{du}{u}=y\sum_{\vbf \in V} \alpha^0(z,\sqrt{y}\vbf) \phi(\vbf) e(Q(\vbf)\tau)\frac{du}{u} \in \Omega^N(X)^\Gamma \otimes C^\8(\HH).
\end{align}
We will be interested in its pushforward
\begin{align}
E_\psi(z_1,\tau,\phi,s) \coloneqq \pi_\ast \left ( \Theta_{\psi}(z,\tau,\phi)u^{-2Ns} \right ) \in \Omega^{N-1}(S)^\Gamma \otimes C^\8(\HH).
\end{align}
\begin{prop}\label{rapid decrease} The theta series $\Theta_{\varphi}(z,\tau,\phi)$, $\Theta_{\alpha}(z,\tau,\phi)$ and $\Theta_{\psi}(z,\tau,\phi)$ are $O(e^{-Cu})$ as $u \rightarrow \8$ and $O(e^{-C/u})$ as $u \rightarrow 0$ (uniformly for $(z,\tau)$ in a compact set). Hence, the integrals $E_{\varphi}(z_1,\tau,\phi,s)$ and $E_{\psi}(z_1,\tau,\phi,s)$ converge uniformly on compact sets and define smooth forms on $S$.
\end{prop}
\begin{proof} It follows from Proposition 4.4 and 4.9 in \cite{rbrsln}. 
By Poisson summation, we have
\begin{align}
\Theta_{\varphi}(z,\tau,\phi) =\sqrt{y}^{-N} \sum_{\vbf \in V} \omega_\8(1,h_\tau)\varphi(z,\vbf) \phi(\vbf) =\sqrt{y}^{-N} \sum_{\vbf \in V} \omega_\8'(1,h_\tau)\widehat{\varphi}(z,\vbf) \vphat(\vbf),
\end{align}
where from Lemma \ref{prop fourier transform} we have
 \begin{align}
\omega_\8'(1,h_\tau)\widehat{\varphi}(z,\vbf)=y^{-N}u^{-2N}i^N e^{-\pi  \lVert u^{-1}g_1^{-1}(v\tau+w) \rVert^2/y} \sum_{\vert \uud \vert=N}  g_1^{-1}(\overline{v\tau+w})^{\uud}  \lambda(\uud).
\end{align}
The exponential function guarantees the rapid decay, except possibly at $(v,w)=(0,0)$. But since the function $\overline{g_1^{-1}(v\tau+w)}^{\uud}$ vanishes at $(0,0)$, we have
\begin{align}
\Theta_{\varphi}(z,\tau,\phi)=O(e^{-C/u}) \ \ \textrm{as} \ \ u \rightarrow 0,
\end{align}
for all $N \geq 1$. Similarly, using the Poisson summation in the first variable shows that
\begin{align}
\Theta_{\varphi}(z,\tau,\phi)=O(e^{-Cu}) \ \ \textrm{as} \ \ u \rightarrow \8.
\end{align}
For the theta series $\Theta_{\alpha}(z,\tau,\phi)$ the idea is the same. The polynomial term in front of, say, $\lambda_{I_l}(\uud)$ is
\begin{align} \label{expression fourier}
y^{-(N-1)}\sum_{\substack{\uud=(d_1, \dots,d_m) \\ d_1+\cdots +d_N=N-1}}\left [ \frac{2\pi \vert \langle g^{-1}(v\tau+w),e_l\rangle \vert ^2+d_ly}{2\pi\langle g^{-1}(\overline{v\tau+w}),e_l\rangle}\right ] \prod_{m=1}^N\langle g^{-1}(\overline{v\tau+w}),e_m\rangle^{d_m}.
\end{align}
For $N>2$, the degree of the product on the right is $N-1\geq 2$. Hence the term vanishes at $(v,w)=(0,0)$.

Now let $N=2$. When $l=1$ and $\uud=(d_1,d_2)=(1,0)$, then the term in front of $\lambda_{I_1}(1,0)$ is
\begin{align}
\frac{\vert \langle g^{-1}(v\tau+w),e_1\rangle \vert ^2}{y}+\frac{1}{2\pi},
\end{align}
which is equal to $1/2\pi$ at $(v,w)=(0,0)$. However, the term $1/2\pi$ is cancelled out by the term $1/2\pi$ from $\lambda_{I_2}(0,1)$; as can be seen in the formulas given in \eqref{formulaalphaxyu}. Thus, when $N=2$ we also have $\alphahat(z,0)=0$.
\end{proof}

\subsection{Regular and singular vectors} A vector $\vbf=(v,w) \in V$ is {\em regular} if $v$ and $w$ are both non-zero, and {\em singular} otherwise. Let $V_{\reg}$ and $V_{\sing}$ be the set of regular (resp. singular) vectors in $V$, so that 
\begin{align}
V=V_{\sing} \sqcup V_{\reg}
\end{align}
and
\begin{align}
\Theta_{\varphi}(z,\tau,\phi)=\sum_{\vbf \in V_{\sing}} \varphi^0(z,\sqrt{y}\vbf) \phi(\vbf) e(Q(\vbf)\tau)+\sum_{\vbf \in V_{\reg}} \varphi^0(z,\sqrt{y}\vbf) \phi(\vbf) e(Q(\vbf)\tau).
\end{align}The sum over the regular vectors is absolutely convergent; see \cite[Lemma.~4.7]{rbrsln}. Hence, we can bring the integral inside the sum and it follows from \eqref{define Ephi} that
\begin{align} \label{fourier Ephi}
E_\varphi(z_1,\tau,\phi,s)& =y^{Ns}e_{\varphi,1}(z_1,\phi,s)+y^{-Ns}e_{\varphi,2}(z_1,\phi,s)  +\sum_{\vbf \in V_\reg} b_\varphi(z_1,\sqrt{y}\vbf,s) \phi(\vbf) e(Q(\vbf)\tau) \qquad \qquad
\end{align}
where
\begin{align}
b_\varphi(z_1,\vbf,s) &\coloneqq (-1)^{N-1}\int_0^\8 \iotau \varphi^0(z,\vbf) u^s \frac{du}{u}, \\
e_{\varphi,1}(z_1,\phi,s) &\coloneqq (-1)^{N-1}  \int_0^\8  \sum_{w \in \Q^N}\iotau \varphi^0(z,(0,w)) \phi(0,w) u^{-2Ns} \frac{du}{u},\\
e_{\varphi,2}(z_1,\phi,s) &\coloneqq (-1)^{N-1}  \int_0^\8 \sum_{v \in \Q^N}\iotau \varphi^0(z,(v,0)) \phi(v,0) u^{-2Ns}\frac{du}{u}.
\end{align}
For the sum over the singular vectors we can bring the integral inside if $\re(s)$ is large enough.

Similarly, we find for $E_\psi$ (we will only need the value at $s=0$)
\begin{align} \label{fourier Epsi}
E_\psi(z_1,\tau,\phi)=ye_\psi(z_1,\phi)+\sum_{\vbf \in V_\reg} y b_\psi(z_1,\sqrt{y}\vbf) \phi(\vbf) e(Q(\vbf)\tau)
\end{align}
where
\begin{align}
b_\psi(z_1,\vbf) \coloneqq \int_0^\8 \alpha^0(z,\vbf) \frac{du}{u}
\end{align}
and $e_\psi(z_1,\phi)=e_{\psi,1}(z_1,\phi)+e_{\psi,2}(z_1,\phi)$ with
\begin{align}
e_{\psi,1}(z_1,\phi) \coloneqq \int_0^\8 \sum_{v \in \Q^N}\alpha^0(z,(v,0))\phi(v,0) \frac{du}{u}, \\
e_{\psi,2}(z_1,\phi) \coloneqq \int_0^\8 \sum_{w \in \Q^N}\alpha^0(z,(0,w))\phi(0,w) \frac{du}{u}.
\end{align}

We will also need Fourier expansions at other cusps.  Let $\Gamma' \backslash \PP^1(\Q)$ be the set of cusps of the congruence subgroup $\Gamma' \subseteq \SL_2(\Z)$. For each $r \in \Gamma' \backslash \PP^1(\Q)$, there is a matrix $\gamma_r \in \SL_2(\Z)$ such that $\gamma_r \8=r$. If $f$ is a smooth function that transforms like a modular form for $\Gamma'$, then $f \vert_{\gamma_r}$ transforms like a modular form for $\gamma_r^{-1}\Gamma' \gamma_r$. It has a Fourier expansion
\begin{align} \label{width denominator}
f \vert_{\gamma_r}=\sum_{n \in \frac{1}{w_r}\Z} a_r\left (n,y \right )e \left ( n \tau\right ),
\end{align}
where $w_r$ is the width of the cusp $r$. We will call it a  {\em Fourier expansion of $f$ at the cusp $r$}.
\begin{rmk} If the width is greater than $1$, the Fourier expansion is not unique. If we replace $\gamma_r$ by $\gamma_r \begin{psmallmatrix} 1 & l \\ 0 & 1 \end{psmallmatrix}$, then the Fourier coefficients would be $a_r(n,y)e \left ( \frac{n l}{w_r}\right )$. However, in the computations of the regularized product, the ambiguity will be cancelled.
\end{rmk}
By \eqref{slash Epsi}, we deduce that $E_\psi(z_1,\tau,\phi) \vert_{\gamma_r}=E_\psi(z_1,\tau,\omega(1,\gamma_r^{-1})\phi)$. Hence, after setting $\phi_r \coloneqq \omega_f(1,\gamma_r^{-1})\phi$, we find that a Fourier expansion at the cusp $r$ is
\begin{align} \label{fourier Epsi p}
E_\psi(z_1,\tau,\phi) \vert_{\gamma_r} =ye_\psi(z_1,\phi_r)+\sum_{\vbf \in V_\reg} y b_\psi(z_1,\sqrt{y}\vbf) \phi_r(\vbf) e(Q(\vbf)\tau).
\end{align}

\subsection{Functional equation} \label{section func}

Let $\Mtil_{N}(\Gamma')$ be the space of smooth functions on $\HH$ that transform like a modular form of weight $N$ for $\Gamma'$ and have moderate growth at all cusps. The Eisenstein class $E_\varphi$ defines a family of lift
\begin{align}
Z_{N-1}(S_\Gamma;\Z) & \longrightarrow \Mtil_{N}(\Gamma'), \\
 c & \longmapsto E_\varphi(c,\tau,\phi,s) \coloneqq \int_c E_\varphi(z_1,\tau,\phi,s)
\end{align}
from cycles to non-holomorphic modular forms of weight $N$. The form $E_\varphi(c,\tau,\phi,s)$ is closed and holomorphic at $s=0$, so that it specializes to a holomorphic lift at the level of cohomology
\begin{align} \label{theta lift at zero}
H_{N-1}(S_\Gamma;\Z) & \longrightarrow M_{N}(\Gamma'),\\
 [c] & \longmapsto E_\varphi(c,\tau,\phi) \coloneqq \int_c E_\varphi(z_1,\tau,\phi).
\end{align}

The lift vanishes if the class $[c]$ is a torsion class. It can also vanish because of a functional equation. We have an involution $g \longmapsto g^\ast \coloneqq g^{-t}$ on $\GL_N(\R)^+$ that induces an involution on $X$ and $S$. Let $\Gamma^\ast$ be the image of $\Gamma$. It induces an involution
\begin{align}
Z_{N-1}(S_\Gamma;\Z) \longrightarrow Z_{N-1}(S_{\Gamma^\ast};\Z), \quad c \longmapsto c^\ast.
\end{align}
Let $\phi^\ast \in \Scal(V_{\A_f})$ be the Schwartz function defined by $\phi^\ast(v,w) \coloneqq \phi(w,v)$. Since $(\rho(\gamma)\phi)^\ast=\rho(\gamma^\ast)\phi^\ast$, the Schwartz function $\phi^\ast$  is preserved by $\Gamma^\ast$.
\begin{thm} \label{functional equation lift}The theta lift satisfies the functional equation
\begin{align}
E_\varphi(c^\ast,\tau,\phi,s)=(-1)^NE_\varphi(c,\tau,\phi^\ast,-s).
\end{align}
\end{thm}
\begin{proof} In the definition of $\varphi(z,\vbf)$ given in \eqref{expression varphi}, we see that for the terms $H_{\uud}(\sqrt{\pi}(g^{-1}v+g^t w))$ and the exponential term, the effect of replacing $g$ by $g^\ast$ is the same as swapping $v$ and $w$. Moreover, the pullback of $\vartheta$ is
\begin{align}
\vartheta^\ast=g^tdg^\ast=((dg^{-1})g)^t=-(g^{-1}dg)^t=-\vartheta^t.
\end{align}
Hence, we have $\lambda^\ast=-\lambda$ and $\lambda(\uud)^\ast=(-1)^N\lambda(\uud)$, thus $\varphi(z,v,w)^\ast=(-1)^N\varphi(z,w,v)$. Finally, replacing $g$ by $g^\ast$ changes $u^s$to $u^{-s}$.
\end{proof}

In particular, if $\phi^\ast=(-1)^N\phi$ then the lift \eqref{theta lift at zero} vanishes on the the subspace of $H_{N-1}(S_\Gamma;\Z)$ that is fixed by the involution.

\section{The regularized lift}

Let $k \in \frac{1}{2}\Z$ with $k \neq 1$. In the cases that we will consider we will have $k \leq 0$ and $k \in \Z$. Let $\Delta_{k}$ be the hyperbolic Laplacian
\begin{align}
\Delta_k \coloneqq -y^2 \left (\frac{\partial^2}{\partial x^2}+\frac{\partial^2}{\partial y^2} \right )+kiy\left (\frac{\partial}{\partial x}+i\frac{\partial}{\partial y} \right ).
\end{align}
 A {\em harmonic weak Maass form} of weight $k$ for the congruence subgroup $\Gamma'\subseteq \SL_2(\Z)$ is a smooth function $f \colon \HH \longrightarrow \C$ that satisfies:
\begin{enumerate}
\item $\restr{f}{\gamma,k}=f \quad \textrm{for all} \ \gamma=\begin{pmatrix}
a & b \\ c & d
\end{pmatrix} \in \Gamma'$,
\item $\Delta_k f=0,$
\item for each cusp $r \in \Gamma' \backslash \PP^1(\Q)$ and $\gamma_r \in \SL_2(\Z)$ such that $\gamma_r \8=r$, there is a polynomial $P_{r} \in \C[q^{-1}]$ such that
\begin{align} \label{growth}
f \vert_{\gamma_r}(\tau)-P_r(\tau)=O(e^{-\epsilon y}) \quad y \rightarrow \8,
\end{align}
for some $\epsilon >0$.
\end{enumerate}
We denote\footnote{In \cite{BF04}, this space is denoted by $H^+_k(\Gamma')$.} by $H_k(\Gamma')$ the space of harmonic weak Maass forms of weight $k$. From the growth condition, it follows that we can write $f=f^++f^-$ where 
\begin{align}
f^+(\tau)& = \sum_{\substack{n \in \Q \\ n \gg -\8}}a^+(f,n)e(n\tau), \\
f^-(\tau)& = \sum_{n \in \Q_{>0}}a^-(f,n)\beta_k(n,y) e(-n\tau),
\end{align}
and
\begin{align}
\beta_k(n,y)=\int_y^\8e^{-4\pi n t}t^{-k}dt.
\end{align}
By \cite[Lemma.~3.4]{BF04}, the Fourier coefficients of $f^-$ satisfy $a^-(n)=O(\vert n \vert^{k/2})$, so that
\begin{align}
f^-(\tau)=O(e^{- C y}) \ \ \textrm{as} \ \ y \rightarrow \8,
\end{align}
for some $C >0$. Similarly, a Fourier expansion at the other cusps is
\begin{align}
f \vert_{\gamma_r}(\tau)& = \sum_{\substack{n \in \Q \\ n \gg -\8}}a^+_r(f,n)e(n\tau) + \sum_{n \in \Q_{>0}} a^-_r(f,n)\beta_k(n,y)e(-n\tau) .
\end{align}
Note that as in \eqref{width denominator}, the rational numbers that appear in the Fourier expansion have denominator $w_r$, the width of the cusp.

The raising and lowering operators $R_k$ and $L_k$ are defined by
\begin{align}
R_k=2i\frac{\partial}{\partial \tau}+ky^{-1}, \quad L_k=-2iy^2\dbar,
\end{align}
and the $\xi_k$ operator by
\begin{align}
\xi_k(f)(\tau) \coloneqq y^{k-2} \overline{L_kf(\tau)}=R_{-k} \left (y^k\overline{f(\tau)} \right ).
\end{align}
The image of a weak Maass form of weight $k$ by $\xi_k$ is a cusp form of weight $2-k$. We have the following exact sequence \cite[Corollary.~3.8]{BF04}
\begin{align}
0 \longrightarrow M^{!}_k(\Gamma') \longrightarrow H_k(\Gamma') \xrightarrow{\ \xi_k\ } S_{2-k}(\Gamma') \longrightarrow 0,
\end{align}
where $M^{!}_k(\Gamma')$ is the space of weakly holomorphic modular forms of weight $k$ (whose poles are supported at the cusps).

\subsection{Regularized lift} Let $D \subset \HH$ be a fundamental domain for $\SL_2(\Z)$ in $\HH$, and 
\begin{align}
F= \bigsqcup_{\gamma \in \Gamma' \backslash \SL_2(\Z)} \gamma D
\end{align}
a fundamental domain for $\Gamma'$. For $T\geq 1$ let $U_T \subset D$ be the subset  consisting of $\tau \in D$ with $\im(\tau) > T$. The complement $D \smallsetminus U_T$ is relatively compact. We can decompose the fundamental domain as 
\begin{align}
F= F_T \cup  \bigsqcup_{\gamma \in \Gamma' \backslash \SL_2(\Z)} \gamma U_T,
\end{align}
where $F_T$ is relatively compact. For a harmonic weak Maass form of weight $2-N$ for $\Gamma'$ we define the regularized pairing 
\begin{align} \label{definition of regularized lift}
\Phi(z_1,f,\phi) \coloneqq \CT \left \{ \lim_{T\rightarrow \8} \int_{F_T} E_{\psi}(z_1,\tau,\phi)f(\tau) y^{-\rho} \dxdy \right \}.
\end{align}
We view the cycle $S_n(\phi)$ as a subset of $S_\Gamma$
\begin{align}
S_n(\phi)= \bigcup_{\substack{ \vbf \in \Gamma \backslash V \\ Q(\vbf)=n \\ \phi (\vbf) \neq 0}} S_{[\vbf]} \subset S_\Gamma.
\end{align}
Consider the set
\begin{align}\label{singular locus}
S_f(\phi)& \coloneqq \bigcup_{r \in \Gamma' \backslash \PP^1(\Q)} \bigcup_{\substack{ n \in \Q_{>0} \\ a^+_r(f,-n) \neq 0}} S_n(\phi_r) \subset S_\Gamma,
\end{align}
where $a^+_r(f,-n)$ are the coefficients of the principal part of $f$ at the cusp $r$, and $\phi_r$ is as in \eqref{fourier Epsi p}. Note that the unions are both finite.

\begin{thm} \label{maintheorem1}The regularized integral converges uniformly on compact subsets of $S \smallsetminus S_f(\phi)$. Hence, it defines a smooth form 
\begin{align}
\Phi(z_1,f,\phi) \in \Omega^{N-1}(S \smallsetminus S_f(\phi))^\Gamma.
\end{align}
\end{thm}

\begin{proof}
The proof follows the idea of the proof of \cite[Proposition.~5.6]{BF04}, with the Kudla-Millson form $\Psi_{\KM}$ being replaced by the form $b_\psi$ that we consider in this paper. We can split the truncated fundamental domain as
\begin{align}
F_T= F_1 \cup  \bigsqcup_{\gamma \in \Gamma' \backslash \SL_2(\Z)} \gamma \left (U_1 \smallsetminus U_T \right),
\end{align}
where $\tau \in U_1 \smallsetminus U_T$ if $1 < \im(\tau) \leq T$. We write

\begin{align}
\Phi(z_1,f,\phi)=\sum_{\gamma \in \Gamma' \backslash \SL_2(\Z)} \widetilde{\Phi}_\gamma(z_1,f,\phi)+\int_{F_1}E_{\psi}(z_1,\tau,\phi)f(\tau) \dxdy
\end{align}
where
\begin{align}
\widetilde{\Phi}_\gamma(z_1,f,\phi) \coloneqq \CT \left \{ \lim_{T\rightarrow \8} \int_{U_1 \smallsetminus U_T} E_{\psi}(z_1,\tau,\phi)\vert_{\gamma} (f \vert_\gamma)(\tau) y^{-\rho} \dxdy \right \}.
\end{align}
Since the integral over $F_1$ converges (it is relatively compact), it is enough to show the convergence of $\widetilde{\Phi}_\gamma(z_1,f,\phi)$. Let us first take $\gamma=1$, and show the convergence of $\widetilde{\Phi}_1(z_1,f,\phi)$. We write $f=f^++f^-$ and decompose the integral as
\begin{align}
\widetilde{\Phi}_1(z_1,f,\phi)=\widetilde{\Phi}_1(z_1,f^+,\phi)+\widetilde{\Phi}_1(z_1,f^-,\phi).
\end{align}The integral
\begin{align}
\widetilde{\Phi}_1(z_1,f^-,\phi)=\lim_{T\rightarrow \8} \int_{U_1 \smallsetminus U_T} E_{\psi}(z_1,\tau,\phi)f^-(\tau) \dxdy
\end{align}
converges by the rapid decrease of $f^-$ and the moderate growth of $E_{\psi}(z_1,\tau,\phi)$ as $y \rightarrow \8$. It is left to show the convergence of $\widetilde{\Phi}_1(z_1,f^+,\phi)$. The domain $U_1 \smallsetminus U_T$ is the rectangle with $-\frac{1}{2} \leq x < \frac{1}{2}$ and $1 <y \leq T$, so that we have
\begin{align}
\lim_{T \rightarrow \8} \int_{U_1 \smallsetminus U_T} E_{\psi}(z_1,\tau,\phi)f^+(\tau)y^{-\rho}\dxdy=\lim_{T \rightarrow \8}  \int_1^T \intx E_{\psi}(z_1,\tau,\phi)f^+(\tau)y^{-\rho} \frac{dxdy}{y^2}.
\end{align}
Inserting the Fourier expansions \eqref{fourier Epsi} gives
\begin{align} \label{intermediate step}
\int_1^\8 \intx  E_\psi(z_1,\tau,\phi)f^+(\tau)y^{-\rho} \dxdy & = a^+(f,0)e_\psi(z_1,\phi) \int_1^\8 y^{-\rho}\frac{dy}{y} \\
& + \int_1^\8\sum_{\vbf \in V_\reg} b_\psi(z_1,\sqrt{y}\vbf) \phi(\vbf) a^+(f,-Q(\vbf))y^{-\rho}\frac{dy}{y}.
\end{align}
The first term does not contribute, since the constant term in the Laurent expansion is
\begin{align} \label{constant term CT}
\CT \left ( \lim_{T\rightarrow \8} \int_1^T \frac{dy}{y^{1+\rho}}\right ) = \CT \left (\frac{1}{\rho} \right )=0.
\end{align}
Recall from \eqref{formalpha} that
\begin{align}
\alpha^0(z,\vbf) &  = P(g^{-1}v,g^t w) e^{-\pi \lVert g^{-1}v-g^t w \rVert^2},
\end{align}
where we set
\begin{align}
P(v,w) \coloneqq -(-1)^{l-1}\frac{1}{2^{N}\pi^{\frac{N-1}{2}}} \sum_{l=1}^N (-1)^{l-1} \langle v-w, e_l \rangle \sum_{\vert \uud \vert=N-1} H_{\uud}(\sqrt{\pi}(v+w)) \lambda_{I_l}(\uud) \in \C[V] \otimes \Omega^N(X).
\end{align}
Note that $P(v,w) \in \C[V] \otimes \Omega^N(X)$ has degree $N$.
Writing $g=g_1 u$, we can use Lemma \ref{lemmapolynomial} to bound
\begin{align}
\vert b_\psi(z_1,\sqrt{y}\vbf) \vert \leq C y^\frac{N}{2} \lVert g_1^{-1}v \rVert^{\frac{N}{2}} \lVert g_1^tw \rVert^{\frac{N}{2}}   \vert P(g_1^{-1}v,g_1^tw)\vert \exp \left (-2\pi y R(z_1,\vbf) \right ),
\end{align}
where
\begin{align}
R(z_1,\vbf) \coloneqq \lVert g_1^{-1}v \rVert \lVert g_1^tw \rVert-Q(\vbf).
\end{align}
By the Cauchy-Schwartz inequality we have
\begin{align}
\lVert g_1^tw \rVert\lVert g_1^{-1}v \rVert \geq \vert \langle v,w \rangle \vert =\vert Q(\vbf) \vert,
\end{align}
where the equality holds if and only if the vectors $g_1^{-1}v$ and $g_1^tw$ are parallel. This is the case if and only if $z_1 \in S_\vbf$. Hence $R(z_1,\vbf)$ is always non-negative, since
\begin{align}
R(z_1,\vbf)= \lVert g_1^tw \rVert\lVert g_1^{-1}v \rVert -\vert Q(\vbf) \vert+\vert Q(\vbf) \vert-Q(\vbf) \geq 0,
\end{align} 
and is zero if and only if $z_1 \in S_\vbf$. (Recall that $S_\vbf$ is empty if $Q(\vbf) < 0$.)

 For $\epsilon>0$ we define the open neighbourhood of $S_\vbf$
\begin{align}
S_\vbf^\epsilon \coloneqq \left \{ z_1 \in S \ \vert \  R(z_1,\vbf) < \epsilon \right \}.
\end{align}
Let $U \subset S$ be a relatively compact open subset and
\begin{align}
A_\epsilon(U,f) \coloneqq \left \{ \vbf \in V_\reg \cap \supp(\phi) \ \vert \ a^+(f,-Q(\vbf)) \neq 0 \ \textrm{and} \ S_\vbf^\epsilon \cap U \neq \emptyset \right \}.
\end{align}
Let $B_\epsilon(U,f) \subset V_\reg \cap \supp(\phi)$ be such that
\begin{align}
V_\reg \cap \supp(\phi) =A_\epsilon(U,f) \sqcup B_\epsilon(U,f).
\end{align}

By Lemma \ref{lemAU}, there is a constant $C_U>0$ such that for all vectors $\vbf \in V_{\reg}$ with $a^+(f,-Q(\vbf)) \neq 0$, we have
\begin{align}
C_U \left (  \lVert \vbf \rVert-M \right ) \leq C_U \left (  \lVert \vbf \rVert-\frac{n_0}{C_U} \right ) \leq  R(z_1, \vbf),
\end{align}
where $M \coloneqq \left \lceil{n_0/C}\right \rceil $. We split the sum as 
\begin{align}
\sum_{\vbf \in B_\epsilon(U,f)}=\sum_{\substack{\vbf \in B_\epsilon(U,f) \\ \lVert \vbf \rVert \leq M   }}+\sum_{\substack{\vbf \in B_\epsilon(U,f) \\ M < \lVert \vbf \rVert}} 
\end{align}
The first sum is a finite sum, so we can bring the integral inside of the sum. Since $\vbf \in B_\epsilon(U,f)$,  we have $R(z_1,\vbf) \geq \epsilon$ and
\begin{align}
b_\psi(z_1,\sqrt{y}\vbf) =  y^\frac{N}{2} O( \exp \left (-2\pi y \epsilon \right )) \quad \textrm{as} \ y\rightarrow \8,
\end{align}
uniformly on $U$. Hence
\begin{align}
 \sum_{\substack{\vbf \in B_\epsilon(U,f) \\ \lVert \vbf \rVert \leq M}} \int_1^\8 b_\psi(z_1,\sqrt{y}\vbf)\phi(\vbf)a^+(f,-Q(\vbf))y^{-\rho} \frac{dy}{y}
\end{align}
converges at $\rho=0$ uniformly on $U$. On the other hand, since the exponent of
\begin{align}
b_\psi(z_1,\sqrt{y}\vbf) =  y^\frac{N}{2} O( \exp \left (-2\pi y C_U \left (  \lVert \vbf \rVert-M \right ) \right )) \quad \textrm{as} \ y\rightarrow \8,
\end{align}
is positive when $M <\lVert \vbf \rVert$, the sum
\begin{align}
\sum_{\substack{\vbf \in B_\epsilon(U,f) \\ M< \lVert \vbf \rVert }}  b_\psi(z_1,\sqrt{y}\vbf)a^+(f,-Q(\vbf))y^{-\rho} \frac{dy}{y}= O( \exp \left (-2\pi y \epsilon \right )) \quad \textrm{as} \ y\rightarrow \8
\end{align}
is rapidly decreasing and the integral converges. Let 
\begin{align}
\Stil_f(\phi)_1 \coloneqq \bigcup_{\substack{ n \in \Q_{>0} \\ a^+(f,-n) \neq 0}} \bigcup_{\substack{ \vbf \in  V \\ Q(\vbf)=n \\ \phi (\vbf) \neq 0}} S_{\vbf}.
\end{align} If $z_1 \in S \smallsetminus \Stil_f(\phi)_1$, then by Lemma \ref{lemAU} we can pick $\epsilon$ and $U$ small enough that $A_\epsilon(U,f)$ is empty, {\em i.e.} $ B_\epsilon(U,f)=V_\reg \cap \supp(\phi)$. Thus, we showed that the form $\widetilde{\Phi}_1(z_1,f,\phi)$ is smooth form on $S$ outside of $\Stil_f(\phi)_1$, which descends to a smooth form on $S_\Gamma$ outside of $S_f(\phi)_1$.

For the other terms, let $r \in \Gamma' \backslash \PP^1(\Q)$ be the cusp $r=\gamma \8$. By plugging-in the Fourier expansion at $r$ given in \eqref{fourier Epsi p}, the right-hand side of \eqref{intermediate step} is
\begin{align}
a^+_r(f,0)e_\psi(z_1,\phi_r) \int_1^\8 y^{-\rho}\frac{dy}{y} + \int_1^\8\sum_{\vbf \in V_\reg} b_\psi(z_1,\sqrt{y}\vbf) \phi_r(\vbf) a^+_r(f,-Q(\vbf))y^{-\rho}\frac{dy}{y}.
\end{align}
In particular, it only depends on the cusp $r=\gamma \8$, and that $\widetilde{\Phi}_{\gamma}(z_1,f,\phi)$ converges outside of the set
\begin{align}
\Stil_f(\phi)_r = \bigcup_{\substack{ n \in \Q_{>0} \\ a^+_r(f,-n) \neq 0}} \bigcup_{\substack{ \vbf \in  V \\ Q(\vbf)=n \\ \phi_r (\vbf) \neq 0}} S_{\vbf}.
\end{align}
The projection of $\bigcup_r \Stil_f(\phi)_r$ in $S_\Gamma$ is precisely $S_f(\phi)$.
 \end{proof}

We also deduce the following from Theorem \ref{maintheorem1}.
\begin{prop}\label{proplim} We have 
\begin{align}
\Phi(z_1,f,\phi) &= \lim_{T\rightarrow \8}  \bigg( \int_{F_T} E_{\psi}(z_1,\tau,\phi)f(\tau) \dxdy - \log(T) Q_\psi(z_1,f,\phi)   \bigg) .
\end{align}
where
\begin{align}
Q_\psi(z_1,f,\phi) \coloneqq \sum_{r \in \Gamma' \backslash \SL_2(\Z) }w_r a^+_r(f,0)e_\psi(z_1,\phi_r),
\end{align}
and $w_r$ is the width of the cusp $r$.
\end{prop}
\begin{proof} Note that $\widetilde{\Phi}_\gamma(z_1,f,\phi)$ only depends on the cusp $r=\gamma \8$, so that
\begin{align}\sum_{\gamma \in \Gamma' \backslash \SL_2(\Z)} \widetilde{\Phi}_\gamma(z_1,f,\phi)=\sum_{r \in \Gamma' \backslash \PP^1(\Q)} w_r\widetilde{\Phi}_{\gamma_r}(z_1,f,\phi),
\end{align}
where $w_r$ is the width of $r$ (it is also equal to the number of translates $\gamma D$ in the fundamental domain $F$, that contain the cusp $r$). The constant term $\CT$ in the definition of $\widetilde{\Phi}_{\gamma_r}(z_1,f,\phi)$ was necessary to remove the divergent constant term
\begin{align}
\lim_{T\rightarrow \8} a^+_r(f,0)e_\psi(z_1,\phi_r) \int_1^T y^{-\rho}\frac{dy}{y},
\end{align}see \eqref{constant term CT}. It follows from
\begin{align}
\lim_{T\rightarrow \8} \int_1^T \frac{dy}{y}=\lim_{T\rightarrow \8} \log(T)
\end{align}
and the rest of the proof of Theorem \ref{maintheorem1}, that the regularization can also be done by substracting 
\begin{align}
\sum_{\gamma \in \Gamma' \backslash \SL_2(\Z)} a^+_r(f,0)e_\psi(z_1,\phi_r) \log(T).
\end{align}
\end{proof}

\subsection{Two lemmas}

\begin{lem}\label{lemmapolynomial} Let $P(\vbf) \in \C[V]$ be a polynomial of degree $N$ and $\vbf=(v,w) \in V_\reg \cap \supp(\phi)$ a regular vector. There is a constant $C>0$ depending on $P$ such that
\begin{align}
\int_0^\8   P(\sqrt{y}u^{-1}v,\sqrt{y}uw) e^{-\pi y \lVert u^{-1}v-uw \rVert^2} \frac{du}{u} \leq  C \sqrt{y}^{N}\lVert v \rVert^{\frac{N}{2}} \lVert w \rVert^{\frac{N}{2}} P(v,w) e^{-2\pi y \left ( \lVert v \rVert \lVert w \rVert-Q(\vbf) \right ) }
\end{align}
for $y \geq 1$.
\end{lem}
\begin{proof}
For $I=\{i_1, \dots,i_N\} \in (\Z_{\geq 0})^N$ we write $v^I \coloneqq v_1^{i_1} \cdots v_N^{i_N}$. We can write $P$ as a sum of monomials
\begin{align}
P(v,w)=\sum_{I,J} a_{I,J}v^Iw^J,
\end{align} 
where the sum is over all $N$-tuples $I,J \in (\Z_{\geq 0})^N$ with $\vert I \vert+\vert J \vert \leq N$.
We have
\begin{align}
& \int_0^\8   P(\sqrt{y}u^{-1}v,\sqrt{y}uw) e^{-\pi y \lVert u^{-1}v-uw \rVert^2} \frac{du}{u} \\
&\qquad = \sum_{I,J} a_{I,J}v^Iw^J \sqrt{y}^{\vert I \vert+\vert J \vert} e^{2 \pi y Q(\vbf)}\int_0^\8  e^{-\pi y u^{-2}\lVert v \rVert^2-\pi yu^2 \lVert w \rVert^2} u^{\vert J \vert-\vert I \vert}\frac{du}{u} \\
&\qquad = \frac{1}{2}\sum_{I,J} a_{I,J}v^Iw^J \sqrt{y}^{\vert I \vert+\vert J \vert}  \left ( \frac{\lVert v \rVert}{\lVert w \rVert} \right )^{\frac{\vert J \vert-\vert I \vert}{2}} e^{2 \pi y Q(\vbf)} K_{\frac{\vert J \vert-\vert I \vert}{2}}(\pi y \Vert v \rVert \lVert w \rVert).
\end{align}
For a fixed real number  $r$ there is a positive constant $C_r$ (depending on $r$) such that for all $a \gg 0$
\begin{align}
K_{r}(a) \leq C_re^{-2a};
\end{align}
see \cite[K7 p.~271]{langelliptic}. Since $\vbf \in V_\reg \cap \supp(\phi) \subset \frac{1}{D}(\Z^N \times \Z^N)$ for some integer $D \geq 1$, we have
\begin{align}
\left ( \frac{\lVert v \rVert}{\lVert w \rVert} \right )^{\frac{\vert J \vert-\vert I \vert}{2}} \leq \frac{\lVert v \rVert^\frac{\vert J \vert}{2} \lVert w \rVert^\frac{\vert I \vert}{2}}{\lVert v \rVert^\frac{\vert I \vert}{2} \lVert w \rVert^\frac{\vert J \vert}{2}}\leq D^N \lVert v \rVert^{\frac{N}{2}} \lVert w \rVert^{\frac{N}{2}}.
\end{align}  
\end{proof}

\begin{lem} \label{lemAU} Let $U \subset S$ be a relatively open compact neigbourhood of a point $z_1 \in S$.
\begin{enumerate}
\item If $\vbf \in V_\reg \cap \supp(\phi)$ and $a^+(f,-Q(\vbf)) \neq 0$, then there is a constant $C_U>0$ that depends only on $U$ and such that
\begin{align}
C_U\lVert \vbf \rVert \leq R(z_1,\vbf)+n_0
\end{align}
for all $z_1 \in U$ and some $n_0 \in \NN$.
\item The set $A_\epsilon(U,f)$ is finite. 
\item If $z_1 \in S \smallsetminus \Stil_f(\phi)_1$, then we can take $\epsilon$ and $U$ small enough that $A_\epsilon(U,f)$ is empty.
\end{enumerate}
\end{lem}
\begin{proof} \begin{enumerate} \item For $g_1 \in G^1$ let 
\begin{align}
\lVert g_1 \rVert \coloneqq \sup_{v \in \R^N} \frac{\lVert g_1v \rVert}{\lVert v \rVert}.
\end{align} Note that $\lVert kv \rVert=\lVert v \rVert$ for $k \in K$, so that $z_1=g_1 K \longmapsto \lVert g_1 \rVert$ is a well-defined continuous function on $S$. Let $D \geq 1$ be an integer such that $\supp(\phi) \subset \frac{1}{D} (\Z^N \times \Z^N)$ and let $C_U>0$ be the constant
\begin{align}
C_U \coloneqq \frac{1}{\sqrt{2}D \sup_{z_1 \in U} \lVert g_1 \rVert \lVert g_1^{-t} \rVert}.
\end{align}
Since the matrix norm satisfies $\lVert A v \rVert \leq \lVert A \rVert \lVert v \rVert$, we have
\begin{align}
R(z_1,\vbf)+Q(\vbf)=\lVert g_1^{-1}v \rVert \lVert g_1^tw \rVert \geq \sqrt{2}D C_U \lVert v \rVert\lVert w \rVert
\end{align} for all $z_1 \in U$ and $\vbf \in V_\reg$. The vector $\vbf$ is contained in $\frac{1}{D} (\Z^N \times \Z^N)$ so that $\lVert v \rVert,\lVert w \rVert \geq 1/D$ and
\begin{align}
 \frac{1}{\sqrt{2}} \lVert \vbf \rVert = \frac{1}{\sqrt{2}} \left (\lVert v \rVert^2+\lVert w \rVert^2 \right )^{\frac{1}{2}} \leq \max(\lVert v \rVert,\lVert w \rVert) \leq D \lVert v \rVert\lVert w \rVert.
\end{align}
Since the principal part of $f^+$ is finite, there is an $n_0 \in \NN$ such that $a^+(f,-n)=0$ if $n > n_0$. Thus, if $a^+(f,-Q(\vbf)) \neq 0$, then $Q(\vbf) \leq n_0$. We deduce that
\begin{align} \label{step in Lemma}
C_U\lVert \vbf \rVert \leq R(z_1,\vbf)+Q(\vbf) \leq R(z_1,\vbf)+n_0.
\end{align}

\item  If $\vbf \in A_\epsilon(U,f)$, then there is a $z_1 \in U$ such that $R(z_1,\vbf)<\epsilon$. By the first part of the lemma, it implies that $\lVert \vbf \rVert$ is bounded by
\begin{align}
\lVert \vbf \rVert \leq \frac{R(z_1,\vbf)+n_0}{C_U} < \frac{\epsilon+n_0}{C_U}.
\end{align} It follows that $A_\epsilon(U,f)$ is finite. 

\item Now suppose that $z_1 \in S \smallsetminus \Stil_f(\phi)_1$. First, note that
\begin{align}
\Stil_f(\phi)_1 = \bigcup_{\substack{ n \in \Q_{>0} \\ a^+(f,-n) \neq 0}} \bigcup_{\substack{ \vbf \in  V \\ Q(\vbf)=n \\ \phi (\vbf) \neq 0}} S_\vbf = \bigcup_{\substack{ \vbf \in V \cap \supp(\phi) \\ a^+(f,-Q(\vbf)) \neq 0}} S_\vbf,
\end{align}
since for vectors with $Q(\vbf)<0$ the set $S_\vbf$ is empty.  Thus, if $z_1 \in S \smallsetminus \Stil_f(\phi)_1$ then for all $\vbf \in V_\reg$ such that $a^+(f,-Q(\vbf)) \neq 0$, we have $R(z_1,\vbf)>0$. The function $R(z_1,\vbf)$ is continuous in $z_1$, so for each of the finitely many vectors $\vbf \in A_\epsilon(U,f)$ we can find a neighbourhood $U_\vbf \subset U$ of $z_1$ such that 
\begin{align}
\inf_{z_1 \in U_\vbf} R(z_1,\vbf)> \epsilon'>0
\end{align}
for some $0<\epsilon' <\epsilon$. Let $U' \subseteq U$ be the neighbourhood
\begin{align}
U' \coloneqq \bigcap_{\vbf \in A_\epsilon(U,f)} U_\vbf.
\end{align}  Since $U' \subseteq U$ and $0<\epsilon' <\epsilon$, we have
\begin{align}
A_{\epsilon'}(U',f) \subseteq A_\epsilon(U,f).
\end{align}
We claim that $A_{\epsilon'}(U',f)$ is empty. Suppose that $\vbf_0 \in A_{\epsilon'}(U',f)$. Then there is $\tilde{z}_1 \in U'$ such that $R(\tilde{z}_1,\vbf_0)<\epsilon'$. On the other hand, since $\tilde{z}_1 \in U' \subset U_{\vbf_0}$, we also have
\begin{align}
R(\tilde{z}_1,\vbf_0)>\inf_{z_1 \in U_\vbf} R(z_1,\vbf_0)> \epsilon',
\end{align}
which contradicts $R(\tilde{z}_1,\vbf_0)<\epsilon'$. This shows that $A_{\epsilon'}(U',f)$ is empty.
\end{enumerate}
\end{proof}

\subsection{Local integrability}

Let $\Delta \subset \R^{N-1}$ be an $(N-1)$-simplex. A chain $c \colon \Delta \longrightarrow S_\Gamma$ is smooth if there exists an open neighbourhood $\Delta \subset U \subset \R^{N-1}$ of the simplex such that the map extends to a smooth immersion $c \colon U \longrightarrow S_\Gamma$. 

When the cycle $c$ intersects $S_\vbf$ transversally in $S$, then the preimage $\bar{c}$ of $c$ in $X$ intersect $X_\vbf$ transversally. In a neighbourhood of an intersection point $X_\vbf \cap \bar{c}$, the restriction of the singular form looks like a differential form on a ball $B_\delta(0) \subseteq \R^N$ with a singularity at most $1/\lVert x \rVert^{N-1}$ at the origin. This explains that the form is integrable over transverse cycles, which is proved in the following proposition.

\begin{thm} \label{proplim2} Let $c \colon \Delta \longrightarrow S_\Gamma$ be a smooth $N-1$ chain that intersects $S_f(\phi)$ transversally. Then the integral of $\Phi(z_1,f,\phi)$ over $c$ converges, and we have
\begin{align}
\int_{c}\Phi(z_1,f,\phi) & = \lim_{T\rightarrow \8} \bigg( \int_{F_T} \int_{c}E_{\psi}(z_1,\tau,\phi)f(\tau) \dxdy -\log(T) \int_cQ_\psi(z_1,f,\phi)  \bigg).
\end{align}
\end{thm}
\begin{proof} Let $\Fcal$ be a fundamental domain for $\Gamma$ is $S$. Let $\pi_\Gamma \colon S \rightarrow S_\Gamma$ be the projection. By subdividing $\Delta$ if necessary, we can write $c$ as a finite sum of cycles $c_i \colon \Delta \longrightarrow S_\Gamma$ such that each $c_i$ is a smooth embedding and the image of $c_i(\Delta)$ is contained in an open subset $U_i \subset S_\Gamma$ that is small enough to be diffeomorphic to $\pi_\Gamma^{-1}(U_i) \subset \Fcal$ in the fundamental domain. The integral over $c$ is then
\begin{align}
\int_{c}\Phi(z_1,f,\phi)=\sum_i \int_{c_i} \Phi(z_1,f,\phi),
\end{align}
where we view $c_i$ as a chain in $\Fcal \subset S$, that intersects each submanifold $S_\vbf \subset \Stil_f(\phi)=\pi^{-1}(S_f(\phi))$ transversally. Let us show that these integrals converge (we will write $c$ instead of $c_i$).

We proceed as in Theorem \ref{maintheorem1} and write
\begin{align}
\Phi(z_1,f,\phi)=\sum_{\gamma \in \Gamma' \backslash \SL_2(\Z)} \widetilde{\Phi}_\gamma(z_1,f,\phi)+\int_{F_1}E_{\psi}(z_1,\tau,\phi)f(\tau) \dxdy
\end{align}
where
\begin{align}
\widetilde{\Phi}_\gamma(z_1,f,\phi) \coloneqq \CT \left \{ \lim_{T\rightarrow \8} \int_{U_1 \smallsetminus U_T} E_{\psi}(z_1,\tau,\phi) \vert_\gamma \left ( f \vert_\gamma \right )(\tau) y^{-\rho} \dxdy \right \}.
\end{align}
The integral over $F_1$ defines a smooth form, so integrable over $c$ and we can interchange the integrals:
\begin{align}
\int_{c}\int_{F_1}E_{\psi}(z_1,\tau,\phi)f(\tau) \dxdy=\int_{F_1}\int_{c} E_{\psi}(z_1,\tau,\phi)f(\tau) \dxdy.
\end{align}Hence, it is enough  to show that $\widetilde{\Phi}_\gamma(z_1,f,\phi)$ is integrable.  We can take $\gamma=1$ without loss of generality. As in Proposition \ref{proplim}, we have \begin{align}
\widetilde{\Phi}_1(z_1,f,\phi) & = \lim_{T\rightarrow \8} \bigg( \int_{U_1 \smallsetminus U_T} E_{\psi}(z_1,\tau,\phi)f(\tau) \dxdy -a^+(f,0)e_\psi(z_1,\phi)\log(T) \bigg).
\end{align}
The integral against $f^-$ is also a smooth form since $f^-$ is rapidly decreasing and $E_\psi$ moderately growing as $y\rightarrow \8$. Thus is enough to show that the integral
\begin{align}
\int_{c} \lim_{T\rightarrow \8} \bigg( \int_1^T\int_{-\frac{1}{2}}^{\frac{1}{2}} E_{\psi}(z_1,\tau,\phi)f^+(\tau) \dxdy -a^+(f,0)e_\psi(z_1,\phi)\log(T) \bigg)
\end{align}
against $f^+$ converges and we can exchange the two integrals.  By plugging-in the Fourier expansion and recalling that $ye_\psi(z_1,\phi)$ is the constant term of $E_{\psi}(z_1,\tau,\phi)$ (see \eqref{fourier Epsi}), we find that the latter integral is equal to
\begin{align} \label{convergence}
\int_{c} \int_1^\8 \sum_{\vbf \in V_\reg} \phi(\vbf)b_\psi(z_1,\sqrt{y}\vbf)a^+(f,-Q(\vbf)) \frac{dy}{y}.
\end{align} By Fubini's theorem, we have
\begin{align}
\int_{c} \int_1^\8 \sum_{\vbf \in V_\reg} \phi(\vbf)b_\psi(z_1,\sqrt{y}\vbf)a^+(f,-Q(\vbf)) \frac{dy}{y} &= \int_1^\8 \int_{c} \sum_{\vbf \in V_\reg} \phi(\vbf)b_\psi(z_1,\sqrt{y}\vbf)a^+(f,-Q(\vbf)) \frac{dy}{y}
\end{align}
where the integral on the left converges absolutely if and only if the integral on the right converges absolutely. Assuming the convergence, it follows from the Fourier expansion
of $\int_{c}E_\psi(z_1,\tau,\phi)$ that
\begin{align}
& \int_1^\8 \int_{c} \sum_{\vbf \in V_\reg} \phi(\vbf) b_\psi(z_1,\sqrt{y}\vbf)a^+(f,-Q(\vbf)) \frac{dy}{y} \\
& \hspace{3cm} = \lim_{T\rightarrow \8} \left ( \int_{U_1 \smallsetminus U_T} \int_{c} E_{\psi}(z_1,\tau,\phi)f^+(\tau) \dxdy -a^+(f,0)\int_{c}e_\psi(z_1,\phi)\log(T) \right ).
\end{align}
This shows that the integral equals the right handside of the proposition.

Finally, we want to show the absolute convergence \eqref{convergence}  by splitting $V_\reg \cap \supp(\phi) = A_\epsilon(U,f) \sqcup B_\epsilon(U,f)$ as in the proof of Theorem \ref{maintheorem1}, where
\begin{align}
A_\epsilon(U,f) \coloneqq \left \{ \vbf \in V_\reg \cap \supp(\phi) \ \vert \ a^+(f,-Q(\vbf)) \neq 0 \ \textrm{and} \ R(z_1,\vbf) < \epsilon \ \textrm{for some} \ z_1 \in U \right \}
\end{align}
is finite and $U$ is a relatively compact open set containing the chain $c$. By Lemma \ref{lemAU}, in particular \eqref{step in Lemma}, if $Q(\vbf)<0$ then $R(z_1,\vbf)$ is bounded by below on $U$. By taking $\epsilon$ small enough, we can assume that $A_\epsilon(U,f)$ only consists of positive regular vectors. For the sum over $B_\epsilon(U,f)$, the absolute convergence follows from the fact  $b(z_1, \vbf)$ is rapidly decreasing as $\lVert \vbf \rVert \rightarrow \8$, uniformly in $z_1 \in U$.

For the (finite) sum over $A_\epsilon(U,f)$ the absolute convergence \eqref{convergence} follows from convergence of the right-handside of
\begin{align} & \int_{c} \int_1^\8 \sum_{\vbf \in A_\epsilon(U,f)} \vert \phi(\vbf)b_\psi(z_1,\sqrt{y}\vbf) a^+(f,-Q(\vbf))\vert \frac{du}{u}\frac{dy}{y} \\
& \hspace{5cm} \leq M \sum_{\vbf \in A_\epsilon(U,f)} \vert a^+(f,-Q(\vbf))\vert \int_{\Delta \times \R_{>0}}  \int_1^\8 \vert c^\ast \alpha^0(z,\sqrt{y}\vbf) \vert \frac{du}{u}\frac{dy}{y},
\end{align}
where $M$ is such that $\vert \phi(\vbf) \vert \leq M$.

 Thus, it is left to show the convergence of the inner integral for a fixed vector $\vbf \in A_\epsilon(U,f)$. We extend the chain  $c \colon \Delta \longrightarrow S$ to a map
 \begin{align}
\bar{c} \colon \Delta \times \R_{>0} \longrightarrow X \qquad \bar{c}(T,u)=u^2c(T).
\end{align}We can view the form inside of the integral as pullback of $ \alpha^0(z,\sqrt{y}\vbf) \frac{du}{u} \in \Omega^{N}(X)$ by $\bar{c}$. To fix some coordinates on $X$, let $B \subset G$ be the Borel subgroup of upper triangular matrices. We have an Iwasawa decomposition $G=NAK$ where $N \subset B$ are the unipotent matrices and $A \simeq \R_{>0}^N$ are diagonal matrices. Let $P \coloneqq NA \subset B$. Since $P \cap K=\{\id\}$ we have a diffeomorphism $P \simeq X$ mapping $g$ to $gg^t$.  Similarly, we have a diffeomorphism $S \simeq P^1 \coloneqq P \cap \SL_N(\R)$. We can view $c$ as a chain in $P^1$, and we will write $g_T \in P$ for the lift of $c(T)$, that satisfies $c(T)=g_Tg_T^t \in S$. We want to compute the integral of the pullback along the map
\begin{align}
\bar{c} \colon \Delta \times \R_{>0} \longrightarrow P \qquad \bar{c}(T,u)=ug_T.
\end{align}
Recall from \eqref{sectionsv} that $X_\vbf$ is the zero locus of the section
\begin{align}
s_\vbf \colon X \longrightarrow G \times_K \R^N, \qquad s_\vbf(g)=\left [g,f_\vbf^-(g) \right ]
\end{align}
where we define
\begin{align}
f_\vbf^\pm (g) = \frac{g^{-1}v\pm g^t w}{2}.
\end{align}
The composition of the two maps is
\begin{align} \label{local diff}
\xi_\vbf^\pm \coloneqq f_\vbf^\pm \circ \bar{c} \colon \Delta \times \R_{>0} \longrightarrow \R^N, \qquad \xi_\vbf^\pm(T,u)=\frac{u^{-1}g_T^{-1}v \pm ug_T^{t}w}{2}.
\end{align}
After identifying $X$ with $P$, we can view $X_\vbf \subset P$ as the zero locus of the smooth map $f_\vbf^-  \colon P \longrightarrow \R^N$.
If $\vbf \in A_\epsilon(U,f)$, then $\vbf$ is regular and positive (for $\epsilon$ small enough).  By Lemma \ref{rank lemma} below, the map $f^-_\vbf$ has rank $N$. Moreover, the submanifold $S_\vbf$ is the image of $X_\vbf$ under the projection $\pi\colon X \rightarrow S$. Since the restriction of the projection to $X_\vbf$ is a diffeomorphism onto $S_\vbf$ and $c$ is transversal to $S_\vbf$, we deduce that the image of $\bar{c} \colon \Delta \times \R_{>0} \longrightarrow P$ is transversal to $X_\vbf$. This means that at point $(T,u) \in \Delta \times \R_{>0}$ such that $p=\bar{c}(T,u) \in X_\vbf$, we have
\begin{align}
T_pX \simeq T_pX_\vbf \oplus \im(d_{(T,u)}\bar{c}) \simeq \ker(d_pf_\vbf^-) \oplus \im(d_{(T,u)}\bar{c}).
\end{align}It implies that the composition $\xi_\vbf^- = f_\vbf^- \circ \bar{c}$ has rank $N$ at any point $(T,u) \in \bar{c}^{-1}(X_\vbf) \cap (\Delta \times \R_{>0})$ and is a local diffeomorphism.

Recall that the form $\alpha^0(z,\vbf)$ was given in \eqref{formalpha} by the formula
\begin{align} 
\alpha^0(z,\vbf)& =-\frac{1}{2^{N}\pi^{\frac{N-1}{2}}} \exp \left (-\pi \lVert g^{-1}v - g^t w \rVert^2 \right )  \\
& \hspace{1cm} \times\sum_{l=1}^N (-1)^{l-1} \langle g^{-1}v-g^t w, e_l \rangle  \sum_{\vert \uud \vert=N-1} H_{\uud}(\sqrt{\pi}(g^{-1}v+g^t w)) \otimes \lambda_{I_l}(\uud).
\end{align}
The pullback of $\alpha^0(z,\vbf)du/u$ by $\bar{c}$ is of the form
\begin{align}
\bar{c}^\ast \left ( \alpha^0(z,\vbf) \frac{du}{u} \right ) \propto \sum_{l=1}^N \exp \left (-\pi \lVert \xi^-(T,u) \rVert^2 \right ) \xi^-_\vbf(T,u)_l  Q_l(\xi^+_\vbf(T,u)_1,\dots,\xi^+_\vbf(T,u)_N)r(T)  \frac{dTdu}{Tu},
\end{align}
up to some constants, where $Q_l$ is some polynomial of degree $N-1$ and $r(T)$ is some smooth function on $\Delta$ (in particular it is bounded). By splitting the polynomial $Q_l$ into polynomials $Q_{l,d}$ of homogeneous degree $0 \leq d \leq N-1$, the convergence of 
\begin{align}
\int_{\Delta \times \R_{>0}}  \int_1^\8   \left \vert \bar{c}^\ast \left ( \alpha^0(z,\sqrt{y}\vbf) \frac{du}{u}  \right ) \right \vert\frac{dy}{y}
\end{align} follows from the convergence of the integral
\begin{align}
& \int_{\Delta \times \R_{>0}} \int_1^\8 \sqrt{y}^{d+1} \exp \left (-\pi y \lVert \xi_\vbf^-(T,u) \rVert^2 \right ) \vert \xi^-_\vbf(T,u)_l \vert  \vert Q_{l,d}(\xi^+_\vbf(T,u)_1,\dots,\xi^+_\vbf(T,u)_N) \vert \frac{dTdudy}{uy} \\
&= \int_{\Delta \times \R_{>0}} \beta_{\frac{d+1}{2}} \left (\pi \lVert \xi^-_\vbf(T,u) \rVert^2 \right ) \vert \xi^-_\vbf(T,u)_l \vert  \vert Q_{l,d}(\xi^+_\vbf(T,u)_1,\dots,\xi^+_\vbf(T,u)_N) \vert \frac{dTdu}{u}
\end{align}
for each $l$, where 
\begin{align}
\beta_{s}(t)=\int_1^\8e^{-ty}y^{s-1}dy=\frac{1}{t^s}\int_t^\8e^{-y}y^{s-1}dy.
\end{align} 
The function
\begin{align}
\beta_{\frac{d+1}{2}} \left (\pi \lVert \xi^-(T,u) \rVert^2 \right )
\end{align}
has singularities at the points $(T,u)$ where
\begin{align}
\xi^-_\vbf(T,u)_1=\dots=\xi^-_\vbf(T,u)_N=0.
\end{align}
Equivalently, the are also the points $(T,u)$ where $v=u^2g_Tg_T^tw$. Since  $T$ lies in the compact set $\Delta$, the determinant $u$ must be bounded and the singularities are all contained in a compact $\Delta \times [a,b]$.

Let us show that the integral converges outside of this compact. Note that we have \cite[Lemma.~3.8]{kudlaintegral}
\begin{align}
\beta_{\frac{d+1}{2}}(r)=O(r^{-\frac{d+1}{2}}) \ \textrm{as} \ r \rightarrow 0, \\
\beta_{\frac{d+1}{2}}(r)=O(e^{-r}) \ \textrm{as} \ r \rightarrow \8.
\end{align}Since $(v,w)$ is regular and $g_T$ is invertible, there is always at least one entry of  $g_T^{-1}v$ and $g_Tw$ that is nonzero. Thus the $\lVert \xi^-(T,u) \rVert^2$ goes to $\8$ as $u$ goes to $0$ and $\8$, and so $\beta_{\frac{d+1}{2}} \left (\pi \lVert \xi^-(T,u) \rVert^2 \right )$ is rapidly decreasing.

Now let us show the convergence in $\Delta \times [a,b]$. Since the polynomials $Q_{d,l}$ are bounded on the compact set $\Delta \times [a,b]$, it is left to show the convergence of
 \begin{align}
\int_{\Delta \times [a,b]} \beta_{\frac{d+1}{2}} \left (\pi \lVert \xi^-_\vbf(T,u) \rVert^2 \right ) \vert \xi^-_\vbf(T,u)_l \vert  \frac{dTdu}{u} \leq \frac{1}{a}\int_{\Delta \times [a,b]}  \frac{\vert \xi^-_\vbf(T,u)_l \vert}{\lVert \xi^-_\vbf(T,u) \rVert^{d+1} } dTdu,
\end{align}
where the set $\Delta \times [a,b]$ contains all the singularities. Around any singularity $p \in \Delta \times [a,b]$, the map $\xi_\vbf^-$ defined in \eqref{local diff} is a local diffeomorphism. Thus, after the change of variables $x_i=\xi_\vbf^-(T,u)_i$ in a small neighbourhood of $(T,u)$, the integral is
\begin{align}
\int_{B_\delta(0)}  \frac{\vert x_l \vert}{\lVert x \rVert^{d+1} }dx_1 \cdots dx_N \leq \int_{B_\delta(0)}  \frac{1}{\lVert x \rVert^{d} }dx_1 \cdots dx_N
\end{align}
in a small neighbourhood of $B_\delta(0) \subset \R^N$ of $0$. Since $0 \leq d \leq N-1$, the integral on the right converges to $\frac{\vol(S^{N-1})}{N-d}\delta^{N-d}$, and thus the integral converges.
\end{proof}

\begin{rmk} One can also show that if $\omega \in \Omega^k_c(S/\Gamma)$ is a compactly supported form of degree $k=\frac{N^2+N}{2}$, then
\begin{align}
\int_{S}\Phi(z_1,f,\phi) \wedge \omega
\end{align} converges. 
\end{rmk}

\begin{lem} \label{rank lemma} If $\vbf$ is a regular vector with $Q(\vbf) > 0$, then $f^-_\vbf$ has rank $N$.
\end{lem}

\begin{proof} We identify the tangent space $T_g\GL_N(\R) \simeq \Mat_{N \times N}(\R)$. The differential at $g$ is
\begin{align}
d_g f_\vbf^- (A)= \left . \frac{d}{dr} \right \vert_{r=0} f_\vbf^- (ge^{rA})=-\frac{1}{2} \left ( Ag^{-1}v+A^tg^tw \right ).
\end{align}
Since $\vbf=(v,w)$ is a regular vector with $Q(\vbf) \geq 0$ if and only if it holds for $\rho_{g^{-1}} \vbf=(g^{-1}v,g^tw)$, it is enough to show that the map \begin{align}
F \colon \Mat_{N \times N}(\R) \longrightarrow \R^N, \quad A \longmapsto Av+A^tw
\end{align}is surjective. This is equivalent to showing that the dual map $ F^\ast \colon (\R^N)^\ast \longrightarrow \Mat_{N \times N}(\R)^\ast$ is injective. We identify $(\R^N)^\ast \simeq \R^N$ via the scalar product $\langle \ , \ \rangle$, so the dual map sends a vector $x \in \R^N$ to $F^\ast(x) \in \Mat_{N \times N}(\R)^\ast$ defined by 
\begin{align}
F^\ast(x)(A)= \langle x, F(A) \rangle=x^tAv+x^tA^tw=x^tAv+w^tAx=\Tr(A(vx^t+xw^t)).
\end{align}
We identify $ \Mat_{N \times N}(\R)^\ast \simeq \Mat_{N \times N}(\R) $ via the trace pairing $(A,B) \mapsto \Tr(AB)$, so that we can see the dual as a map
\begin{align}
F^\ast \colon \R^N \longrightarrow \Mat_{N \times N}(\R) , \quad x \longmapsto vx^t+xw^t.
\end{align}
To show that the map is injective, suppose that $x$ is such that $vx^t+xw^t=0$. Multiplying by the nonzero vector $w$ (since $\vbf$ is regular) shows that $x= \alpha v$, where $\alpha = -(x^tw) \lVert w \rVert^{-2}$. We have to show that $\alpha=0$. After substituting in the previous equation, we get
\begin{align}
0=vx^t+xw^t=\alpha (\lVert v \rVert^2+ Q(v,w)).
\end{align}
Since $v$ is nonzero and $Q(\vbf)=Q(v,w) > 0$, it follows that $\alpha=0$.
\end{proof}

\subsection{Adjointness formula}

When the lift \eqref{theta lift at zero} vanishes, we want to consider the derivative
\begin{align}
 Z_{N-1}(S_\Gamma;\Z) & \longrightarrow \Mtil_{N}(\Gamma') \\
 c & \longmapsto E'_\varphi(c,\tau,\phi) \coloneqq \dds \int_c E_\varphi(z_1,\tau,\phi,s).
\end{align}
From the Fourier expansion \eqref{fourier Ephi} we find that
\begin{align}
E'_\varphi(c,\tau,\phi)=A(c,\phi)+\log(y)B(c,\phi)+\sum_{n \in \Q} C_n(c,\phi,y)  e(n\tau)
\end{align}
where
\begin{align}
A(c,\phi) & = \dds \int_c \left ( e_{\varphi,1}(z_1,\phi,s) + e_{\varphi,2}(z_1,\phi,s) \right ), \\
B(c,\phi) & =N\int_c \left ( e_{\varphi,1}(z_1,\phi) - e_{\varphi,2}(z_1,\phi) \right ), \\
C_n(c,\phi,y) & = (-1)^{N-1} \sum_{\substack{\vbf \in V_\reg \\ Q(\vbf)=n}} \phi(\vbf) \int_c \int_0^\8 \iotau \varphi^0(z,\sqrt{y}\vbf) \log(u) \frac{du}{u}.
\end{align}
At the other cusps we have
\begin{align}
E'_\varphi(c,\tau,\phi)\vert_{\gamma_r}=A(c,\phi_r)+\log(y)B(c,\phi_r)+\sum_{n \in \Q} C_n(c,\phi_r,y)  e(n\tau).
\end{align}

\begin{prop} \label{lemmaderiv} Let $c \in Z_{N-1}(S_\Gamma)$ be a smooth cycle. Then
\begin{align}
\int_{c}E_{\psi}(z_1,\tau,\phi)= \frac{(-1)^{N-1}}{2N} L_N E'_\varphi(c,\tau,\phi).
\end{align}
\end{prop}
\begin{proof} 
Recall from \eqref{transgressionformula} that $d\Theta_{\alpha}(z,\tau,\phi)=L_N \Theta_{\varphi}(z,\tau,\phi)$. We have \begin{align}
d_X \biggl ( \Theta_{\alpha}(z,\tau,\phi) \log(u)\biggr )=L_N\Theta_{\varphi}(z,\tau,\phi)\log(u)+(-1)^{N-1}\Theta_{\alpha}(z,\tau,\phi)\frac{du}{u}
\end{align}
and
\begin{align}
\dds L_N \left (\Theta_{\varphi}(z,\tau,\phi)u^{-2Ns}\right )=-2NL_N\Theta_{\varphi}(z,\tau,\phi)\log(u).
\end{align}
For a form $\eta$ on $X$ that vanishes as $u$ goes to $0$ and $\8$ (and is integrable along the fiber of $X \rightarrow S$) we have $d_S\pi_\ast(\eta)=\pi_\ast(d_X\eta)$ (see \cite[Lemma.~3.9]{rbrsln}). By Proposition \ref{rapid decrease} the form $\Theta_{\alpha}(z,\tau,\phi)$ is rapidly decreasing as $u \rightarrow 0$ and $\8$, and so the same holds for $\Theta_{\alpha}(z,\tau,\phi)\log(u)$.
Thus, we find that
\begin{align}
d_S\pi_\ast \left ( \Theta_{\alpha}(z,\tau,\phi) \log(u) \right )=-\frac{1}{2N}\dds L_N \pi_\ast \left (\Theta_{\varphi}(z,\tau,\phi)u^{-2Ns}\right )+(-1)^{N-1}E_{\psi}(z_1,\tau,\phi).
\end{align}
After integrating over the cycle $c$ we find that
\begin{align}
0=-\frac{1}{2N}L_N \dds \int_c E_{\varphi}(z_1,\tau,\phi,s)+(-1)^{N-1}\int_{c}E_{\psi}(z_1,\tau,\phi).
\end{align}
Note that the integral and the sum appearing in the definition of $E_\psi(z_1,\tau,\phi,s)$ and $E_\varphi(z_1,\tau,\phi,s)$ converge uniformly for $(z_1,\tau,s)$ in a compact set of $S \times \HH \times \C$, (as well as their derivatives) so that the exchange of the differential operators with the sum and integral is justified.
\end{proof}

\begin{thm}\label{main formula 1}Let $f \in H_{2-N}(\Gamma')$ be a harmonic weak Maass form and $g=\xi_{2-N}(f) \in S_N(\Gamma')$ its image under the $\xi$-operator. For any smooth cycle $c$ transverse to $S_f(\phi)$ we have 
\begin{align}
2N (-1)^{N-1}\int_{c}\Phi(z_1,f,\phi) =  \langle E'_\varphi(c,\phi),g \rangle+ \sum_{ r \in \Gamma' \backslash \PP^1(\Q)}w_r \kappa_r(c,f,\phi)
\end{align}
where
\begin{align}
\kappa_r(c,f,\phi) & \coloneqq A(c,\phi_r)a^+_r(f,0)  + \lim_{T \rightarrow \8}\sum_{n \in \Q_{>0}}C_n(c,\phi_r,T)  a^+_r(f,-n).
\end{align}
\end{thm}
\begin{proof} 
From Proposition \ref{lemmaderiv} we have
\begin{align}
\int_{c}E_{\psi}(z_1,\tau,\phi)= \frac{(-1)^{N-1}}{2N} L_N E'_\varphi(c,\tau,\phi).
\end{align}
In particular, by comparing with the constant term of $E_\psi$ given in \eqref{fourier Epsi} and using that $L_N\log(y)=y$ we find that
\begin{align}
\int_ce_\psi(z_1,\phi)= \frac{(-1)^{N-1}}{2N} B(c,\phi).
\end{align}
Combining with Theorem \ref{proplim2}, we deduce that
\begin{align} \label{intermediate equality}
\int_{c}\Phi(z_1,f,\phi) & = \lim_{T\rightarrow \8} \bigg( \int_{c}E_{\psi}(z_1,\tau,\phi) \dxdy -\int_cQ_\psi(z_1,f,\phi) \log(T) \bigg) \\
& = \frac{(-1)^{N-1}}{2N} \lim_{T\rightarrow \8} \bigg( \int_{F_T} L_N E'_\varphi(c,\tau,\phi) \dxdy - \log(T) \sum_{r \in \Gamma' \backslash \PP_1(\Q)}w_r a^+_r(f,0) B(c,\phi_r) \bigg).
\end{align}
Since $f$ is of weight $2-N$ and $E'_\varphi(c,\tau,\phi)$ of weight $N$, the form $E'_\varphi(c,\tau,\phi)f(\tau)d\tau \in \Omega^1(\HH)$ is a $1$-form on $\HH$. Applying the exterior derivative $d_\HH$ on both sides gives
\begin{align}
d_\HH \left (E'_\varphi(c,\tau,\phi)f(\tau)d\tau \right )=-(L_NE'_\varphi(c,\tau,\phi))f(\tau)\dxdy+E'_\varphi(c,\tau,\phi)\overline{\xi_{2-N}f(\tau)}y^N\dxdy.
\end{align}
By Stoke's Theorem we have
\begin{align}
\int_{\partial F_T} E'_\varphi(c,\tau,\phi)f(\tau)d\tau & =\int_{F_T}d_\HH \left (E'_\varphi(c,\tau,\phi)f(\tau)d\tau \right ) \\
& =-\int_{F_T}(L_NE'_\varphi(c,\tau,\phi))f(\tau)\dxdy+\int_{F_T} E'_\varphi(c,\tau,\phi)\overline{\xi_{2-N}f(\tau)}y^N\dxdy.
\end{align}
Hence, the limit in \eqref{intermediate equality} is equal to
\begin{align}
 &\langle E'_\varphi(c,\phi),g \rangle+\lim_{T\rightarrow \8} \bigg ( -\int_{\partial F_T} E'_\varphi(c,\tau,\phi)f(\tau) d\tau - \log(T) \sum_{r \in \Gamma' \backslash \PP_1(\Q)}w_r a^+_r(f,0) B(c,\phi_r) \bigg ) \\
&  = \langle E'_\varphi(c,\phi),g \rangle+\sum_{r \in \Gamma' \backslash \PP_1(\Q)} \lim_{T\rightarrow \8} \bigg( \int_{-\frac{w_r}{2}}^\frac{w_r}{2} E'_\varphi(c,x+iT,\phi_r)f\vert_{\gamma_r}(x+iT) dx- w_r \log(T)a^+_r(f,0) B(c,\phi_r) \bigg).
\end{align}
Note that in the second equality we used that the boundary $\partial F_T$ of the truncated fundamental domain is oriented counterclockwise, and that the sides cancel out since $E'_\varphi(c,\tau,\phi)f(\tau)$ is $\Gamma'$-invariant. The limit
\begin{align}
\lim_{T \rightarrow \8} \int_{-\frac{w_r}{2}}^\frac{w_r}{2} E'_\varphi(c,x+iT,\phi_r)(f\vert_{\gamma_r})^-(x+iT) dx=0
\end{align}
vanishes because $(f\vert_{\gamma_r})^-$ is rapidly decreasing as $T \rightarrow \8$. For the integral against $(f\vert_{\gamma_r})^+$ we plug-in the Fourier expansions. Using that 
\begin{align}
\int_{-\frac{w_r}{2}}^\frac{w_r}{2} e^{2i \pi x \frac{m+n}{w_r}}=w_r\delta_{m+n=0}
\end{align}
for two integers $m,n$, we find that
\begin{align}
&\int_{-\frac{w_r}{2}}^\frac{w_r}{2} E'_\varphi(c,x+iT,\phi_r)(f\vert_{\gamma_r})^+(x+iT)dx- w_r\log(T)a^+_r(f,0)B(c,\phi_r) \\
& \hspace{3cm}=  w_rA(c,\phi_r)a^+_r(f,0)  +w_r \lim_{T \rightarrow \8}\sum_{n \in \Q_{>0}}C_n(c,\phi_r,T)  a^+_r(f,-n).
\end{align}
\end{proof}

\section{Period over a torus}
 Let $F$ be a totally real field with ring of integers $\Ocal$ and $\sigma_1,\dots,\sigma_N$ the $N$ real embeddings. Let $F^{\times,+}$ be the set of totally positive elements in $F$, and $\Ocal^{\times,+} = \Ocal^{\times} \cap F^{\times,+}$ the totally positive units. Let $W \coloneqq F$, that we view as a $\Q$-vector space. Let $\epsilon=(\epsilon_1, \dots, \epsilon_N)$ be a $\Q$-basis of $F$, chosen such that
\begin{align}
g_{\epsilon} \coloneqq \begin{pmatrix}
\sigma_1(\epsilon_1) & \cdots & \sigma_1(\epsilon_N) \\[1em]
\vdots & & \vdots \\[1em]
\sigma_N(\epsilon_1) & \cdots & \sigma_N(\epsilon_N)
\end{pmatrix} \in \GL_N(\R)^+
\end{align}
has positive determinant. 

 \subsubsection{ A finite Schwartz function} Let $p$ be a prime, and $\afrak,\cfrak$ be two integral ideals such that $\cfrak \mid p$. With respect to the basis $\epsilon$, we indentify $F \simeq \Q^N$. Then $L_1=\afrak \cfrak$ and $L_2=\afrak$ are two lattices in $\Q^N$ that satisfy $pL_2 \subset L_1 \subset L_2 \subset \Q^N$, and $L_2^\vee=\afrak^{-1}\dfrak^{-1}$. Let $\phi_{\afrak,\cfrak}$ be the finite Schwartz function $\phi_{\afrak,\cfrak} \coloneqq \prod_p 1_{L_1 \otimes \Z_p \times L_2^\vee\otimes \Z_p }$ is in $\Scal(V_{\A_f})$. More generally, for a totally odd character $\chi \colon \Cl(F)^+ \longrightarrow \{ \pm 1\}$ we define
\begin{align}
\phi_{\chi,\cfrak} \coloneqq \sum_{\afrak \in \Cl(F)^+} \chi(\afrak)\phi_{\afrak,\cfrak} \in \Scal(V_{\A_f}).
\end{align}
Let $\Gamma \subseteq \SL_N(\Z)$ be a torsion-free subgroup that preserves $\phi_{\chi,\cfrak}$ for all $\cfrak \mid p$ under the $\rho$-action, {\em i.e.}
\begin{align}
\phi_{\chi,\cfrak}(\gamma^{-1}v,\gamma^tw)=\phi_{\chi,\cfrak}(v,w) \qquad \forall \gamma \in \Gamma.
\end{align}
 Let $R_\epsilon \colon F^\times \hooklongrightarrow \GL_N(\Q)$ be the regular representation with respect to the basis $\epsilon$ of $F$. Since $\afrak \cfrak$ is preserved by multiplication with $\Ocal^\times$, we have $R_\epsilon(\Ocal^\times) \subset \SL_N(\Q)$. Let $\Lambda_\epsilon\coloneqq R_\epsilon^{-1}(R_\epsilon(\Ocal^{\times,+}) \cap \Gamma)$.

\subsubsection{The cycle attached to $F$} Let $A_\R \coloneqq \R_{>0}^N$, that we view as a subgroup of $G$ via the map
\begin{align}
t=(t_1, \dots,t_N) \in \R_{>0}^N \longmapsto a(t)=\begin{pmatrix} t_1 &  & 0 \\  & \ddots &  \\ 0 &  & t_N \\ \end{pmatrix} \in G.
\end{align}Let $A^1_\R \coloneqq A_\R \cap G^1$ be the subgroup of elements of norm $1$.  The group $F^{\times,+}$ acts on $A_\R$ by 
\begin{align}
\nu (t_1, \dots,t_N)=(\sigma_1(\nu)t_1, \dots, \sigma_N(\nu)t_N),
\end{align}
and the group $F^{1,+}$ of elements of norm $1$ acts on $A^1_\R$. Since $\Lambda_\epsilon\coloneqq R_\epsilon^{-1}(R_\epsilon(\Ocal^{\times,+}) \cap \Gamma)$ has finite index in $\Ocal^{\times,+}$, it follows from Dirichlet's unit theorem that
\begin{align}
\Lambda_\epsilon \backslash A^1_\R
\end{align}
is compact.
The regular representation is diagonalized by $g_{\epsilon}$, so that for all $\nu \in F^\times$ we have
\begin{align}
R_\epsilon(\nu)=g^{-1}_{\epsilon}\begin{pmatrix} \sigma_1(\nu) &  & 0 \\  & \ddots &  \\ 0 &  & \sigma_N(\nu) \\ \end{pmatrix}g_{\epsilon}.
\end{align}
Thus, the map $c_\epsilon(t) \coloneqq g_{\epsilon}^{-1}a(t)$ is $\Ocal^{\times,+}$-equivariant and induces a smooth map
\begin{align}
c_\epsilon \colon \Lambda_\epsilon \backslash A^1_\R \longrightarrow \Gamma \backslash S.
\end{align}
It defines a smooth cycle $c_\epsilon \in Z_{N-1}(S_\Gamma,\Z)$.
\begin{prop} The cycle $c_\epsilon$ is transverse to any $S_{[\vbf]}$ with $\vbf \in V_\reg$ and $Q(\vbf) > 0$. 
\end{prop}
\begin{proof}  By translating by $g_{\epsilon}$, we see that the image of $g_{\epsilon}^{-1}A^1_\R$ is transverse to $S_\vbf$in $S$ for $\vbf \in V_\reg \subset \Q^{2N}$ if the image of $A^1_\R$ is transverse to $S_\wbf$ for $\wbf=(g_{\epsilon}v,g_{\epsilon}^\vee w) \in \sigma(F) \times \sigma(F)^\vee \subset V_\R$. Note that a vector $\wbf \in \sigma(F) \times \sigma(F)^\vee \subset V_\R$ has nonzero entries if it is regular (because $\sigma_i(\lambda)=0$ implies that $\lambda=0$). As in the proof of Theorem \ref{proplim2}, $A^1_\R$ is transverse to $S_\wbf$ if and only $A_\R$ is transverse to $X_\wbf$, which follows from the fact that the restriction of $s_\wbf$ to  the image $A_\R \subset X$ has rank $N$. A direct computation shows that this is the case for a vector $(g_{\epsilon}v,g_{\epsilon}^\vee w)$, since such a vector has nonzero entries if $\vbf$ is a regular.
\end{proof}

\subsection{Hecke characters and $L$-function} We briefly recall the main properties of Hecke $L$-functions for totally real fields, following \cite{neukirch}.
Let $I_F$ be the group of fractional ideals of $F$, and 
\begin{align}
\chi \colon I_F \longrightarrow \{ \pm 1\}
\end{align} be a totally odd Hecke character. {\em i.e.} such that
\begin{align}
\chi((\alpha))=\sgn(N(\alpha))
\end{align}
on principal ideals. The existence of such a character implies that $\Ocal$ does not contain a unit of negative norm, in particular it implies that $F$ has even degree. These are Dirichlet characters of type $(q,0)$ with $q=(1,\dots,1)$ in the notation of \cite[Definition.~6.8 on p.~478]{neukirch}.
Hence, it defines a character $\chi \colon \Cl(F)^+ \longrightarrow \{ \pm 1\}$ on the narrow class group. Let
\begin{align}
 L(\chi,s) \coloneqq \sum_{\substack{\afrak \subset \Ocal}} \frac{\chi(\afrak)}{N(\afrak)^s} \qquad \re(s) \geq 1,
\end{align}
be the associated $L$-function. The completed $L$-function
\begin{align}
\vert D_F \vert^\frac{s}{2}  \Gamma\left (\frac{1+s}{2} \right )^N L(\chi,s)
\end{align}
admits an analytic continuation to the entire plane and a functional equation - see \cite[Theorem.~8.5 on p.~502]{neukirch}.  In particular,  since the Gamma function is nonzero and has poles at negative integers, we deduce that $L(\chi,s)$ is analytic on $\C$ and has a zero of order $N$ at odd negative integers. More generally, the proof of the following Theorem due to Shintani can be found in \cite[Theorem.~9.8 on p.~515]{neukirch}.
\begin{thm}[Shintani] \label{Shintani}  For $n \geq 1$, the special value $L(\chi,1-n)$ is a rational number.
\end{thm}

Let $\afrak_1, \dots, \afrak_h$ be representatives of the narrow class group. If we write $\afrak=(\mu) \afrak_i^{-1}$ for some principal ideal $(\mu)$ then we can split the sum as
\begin{align}
 L(\chi,s)=\sum_{[\afrak_i]\in \Cl(F)^+} \chi(\afrak_i) N(\afrak_i)^s\sum_{\mu \in \afrak_i /\Ocal^{\times,+}} \frac{\sgn(N(\mu))}{\vert N(\mu)\vert^s}.
\end{align}

\subsection{The integral} 
For a totally odd character $\chi \colon \Cl(F)^+ \longrightarrow \{ \pm 1\}$ as above we set
\begin{align}
\Ecal_{\chi,\cfrak}(\tau,s) \coloneqq \sum_{[\afrak] \in \Cl(F)^+} \chi(\afrak) \Ecal_{\afrak,\cfrak}(\tau,s),
\end{align}
where
\begin{align}\label{intermediate label}
\Ecal_{\afrak,\cfrak}(\tau,s) & \coloneqq   \frac{\N(\afrak)^{2s}}{[\Ocal^{\times,+}:\Lambda_\epsilon]}\int_{c_\epsilon} E_{\varphi}(z_1,\tau,\phi_{\afrak,\cfrak},s).
\end{align}
The factor $\N(\afrak)^{2s}$ ensures that the integral only depends on the class of $\afrak$ in $\Cl(F)^+$. It converges for any $s \in \C$ and let $\Ecal_{\chi,\cfrak}(\tau)$ be the value at $s=0$. We have
\begin{align} \Ecal_{\afrak,\cfrak}(\tau,s) = 2^N\N(\afrak)^{2s} \sqrt{y}^{-N} \int_{\Ocal^{\times,+} \backslash A^1_\R}\int_0^\8 \sum_{\vbf \in L_1 \times L_2^\vee } \omega(1,h_\tau)   c_\epsilon^\ast \varphi(z,\vbf) u^{-2Ns}\frac{du}{u}.
\end{align}
By the equivariance of $\varphi(z,\vbf)$ we have
\begin{align}
c_\epsilon^\ast \varphi(z,\vbf)=a^\ast (g^{-1}_{\epsilon})^\ast \varphi(z,\vbf) =a^\ast \varphi(z,\rho_{g_{\epsilon}} \vbf).
\end{align}
The embedding $\sigma(\afrak) \coloneqq \left \{ (\sigma_1(\lambda),\dots,\sigma_N(\lambda)) \in \R^N \ \vert \ \lambda \in \afrak \right \} \subset \R^N$ of $\afrak$ in $\R^N$ via the $N$ real embeddings is a lattice in $\R^N$. We have
 \begin{align} \label{translatelattice}
\rho(g_{\epsilon})(L_1 \times L_2^\vee)=g_{\epsilon} L_1 \times g_{\epsilon}^\vee L_2^\vee= \sigma(\afrak) \times \sigma(\afrak^{-1} \dfrak^{-1}) \subset \R^N \times \R^N.
\end{align}
We will simply write $\afrak \subset \R^N$ for such a lattice, instead of $\sigma(\afrak)\subset \R^N$. By using the diffeomorphism $A_\R \simeq A^1_\R \times \R_{>0}$ we deduce that
\begin{align}
\Ecal_{\afrak,\cfrak}(\tau,s) = \N(\afrak)^{2s} \sqrt{y}^{-N} \int_{\Ocal^{\times,+} \backslash A_\R} \sum_{\vbf \in \afrak\cfrak \times \afrak^{-1} \dfrak^{-1}} \omega(1,h_\tau)   a^\ast \varphi(z,\vbf) u^{-2Ns}\frac{du}{u}.
\end{align}
The restriction of the $1$-form $\lambda$ (see \eqref{form lambda}) to $A_\R$ is
\begin{align}
a^\ast (\lambda) =\begin{pmatrix} \frac{dt_1}{t_1} & \cdots & 0 \\  & \ddots &  \\ 0 & \cdots & \frac{dt_N}{t_N}
\end{pmatrix}.
\end{align}
Hence, the restriction of $\lambda(\uud)$ is only nonzero for $\uud=(1,1,\dots,1)$ and
\begin{align}
\lambda(\uud)=\frac{dt_1}{t_1} \wedge \cdots \wedge \frac{dt_N}{t_N}.
\end{align}
Using the formula given in \eqref{formula varphi zero}, the pullback of $\varphi(z,\vbf)$ to the split torus $A_\R$ is
\begin{align}
\prod_{k=1}^N (v_kt^{-1}_k+w_kt_k)\exp \left (-\pi \left ( v_kt^{-1}_k\right )^2 -\pi \left (w_kt_k \right )^2 \right ) \frac{dt_k}{t_k},
\end{align}
By Proposition \ref{prop fourier transform}, the partial Fourier transform of this form in the variable $w$ is
\begin{align} \label{partial fourier phi}
 i^N \overline{\N(vi+w)}\prod_{k=1}^N \exp \left (-\pi t_k^{-2}\vert v_ki+w_k \vert^2 \right )t_k^{-2} \frac{dt_k}{t_k}.
\end{align}
Applying Poisson summation, unfolding the integral (for $\re(s)$ large enough the sum is absolutely convergent) and translating by $g_{\epsilon}^{-1}$ gives
\begin{align}
\Ecal_{\afrak,\cfrak}(\tau,s) & =\vert D_F \vert^\frac{1}{2}N(\afrak)^{1+2s} \sqrt{y}^{-N} \int_{\Ocal^{\times,+} \backslash A_\R} \sum_{(v,w) \in \afrak\cfrak \times \afrak } \omega'(1,h_\tau)a^\ast \widehat{\varphi}(t,v,w)u^{-2Ns} \\
& = \vert D_F \vert^\frac{1}{2}N(\afrak)^{1+2s} i^N \frac{\overline{\N(v\tau+w)}}{y^N} \sum_{(v,w) \in \afrak\cfrak \times \afrak/\Ocal^{\times,+}} \int_{A_\R} \prod_{k=1}^N  \exp \left (-\pi t_k^{-2}\frac{ \vert v_ki+w_k \vert^2}{y} \right )t_k^{-2-2s} \frac{dt_k}{t_k} \\
& = 2^{-N}i^N \pi^{-N(1+s)}\Gamma \left ( 1+s \right )^N \vert D_F \vert^\frac{1}{2}  \sum_{(v,w) \in \afrak\cfrak \times \afrak/\Ocal^{\times,+}} \frac{y^{Ns}}{\N(v\tau+w) \vert \N(v\tau+w) \vert^{2s}}.
\end{align}
Hence, we find that
\begin{align}
\Ecal_{\chi,\cfrak}(\tau,s)= C(s) E_{\chi,\cfrak}(\tau,s)
\end{align}
where we define
\begin{align} \label{Eisenstein series one}
E_{\chi,\cfrak}(\tau,s) \coloneqq \sum_{[\afrak] \in \Cl(F)^+}\chi(\afrak) N(\afrak)^{1+2s} \sum_{(v,w) \in \afrak\cfrak \times \afrak /\Ocal^{\times,+}} \frac{y^{Ns}}{\N(v\tau+w) \vert \N(v\tau+w) \vert^{2s}}
\end{align}
and $C(s) \coloneqq 2^{-N}i^N \pi^{-N(1+s)}\Gamma \left ( 1+s \right )^N \vert D_F \vert^\frac{1}{2}$.

\subsection{Fourier expansion}We can use the integral representation to find the Fourier expansion. When $N=2$, the computation are similar to \cite{GKZ87}. Using the previous computations \eqref{intermediate label} we deduce that
\begin{align}
& \Ecal_{\chi,\cfrak}(\tau,s) = \sum_{[\afrak] \in \Cl(F)^+} \chi(\afrak) \N(\afrak)^{2s}  \sqrt{y}^{-N} \int_{\Ocal^{\times,+}\backslash A_\R} \sum_{(v,w) \in \afrak\cfrak \times \afrak^{-1} \dfrak^{-1} } \omega(1,h_\tau)c_\epsilon^\ast  \varphi(z,v,w) u^{-2Ns} \\
& =  \sum_{[\afrak] \in \Cl(F)^+} \chi(\afrak) \N(\afrak)^{2s}\int_{\Ocal^{\times,+}\backslash A_\R} \sqrt{y}^{N} \sum_{(v,w)} \prod_{k=1}^N (v_kt^{-1}_k+w_kt_k)t_k^{-2s} e^{-y\pi \left ( v_kt^{-1}_k \right )^2-y\pi \left (w_kt_k \right )^2} \frac{dt_k}{t_k} e(x \tr(vw))
\end{align}
where $Q(\vbf)=\sum_k v_kw_k=\tr(vw)$. We find that the constant term of $\Ecal_{\chi,\cfrak}(\tau,s)$ is given by the singular vectors ($v=0$ or $w=0$) and is equal to
\begin{align} \label{constant term in s}
c_{\chi,\cfrak}(0,2s) \coloneqq 2^{-N}y^{Ns}\chi(\cfrak)N(\cfrak)^{-2s}\Lambda(\chi,2s)+2^{-N}y^{-Ns}\chi(\dfrak)\N(\dfrak)^{-2s}\Lambda(\chi,-2s),
\end{align}
where we set
\begin{align}
\Lambda(\chi,2s) \coloneqq\Gamma \left (\frac{1}{2}+s \right )^N\pi^{-N\left (\frac{1}{2}+s \right )}L(\chi,2s).
\end{align} 

The higher Fourier coefficients are given by the sum over the regular vectors ($v \neq 0$ and $w \neq 0$). After unfolding we find that
\begin{align}
 \Ecal_{\chi,\cfrak}(\tau,s)=c_{\chi,\cfrak}(0,2s)+\sum_{\substack{\nu \in \cfrak\dfrak^{-1} \\ \nu \neq 0}} c_{\chi,\cfrak}(\nu,2s) e(\tau \tr(\nu))
\end{align}
where
\begin{align}
c_{\chi,\cfrak}(\nu,2s) & \coloneqq e^{-2\pi y \Tr(\nu)}\sqrt{y}^{N}\sum_{[\afrak] \in \Cl^+(F)} \chi(\afrak) \N(\afrak)^{2s} \\
& \hspace{3cm} \times \sum_{\substack{(v,w) \in \afrak\cfrak \times \afrak^{-1} \dfrak^{-1}/\Ocal^{\times,+} \\ \nu=vw }} \prod_{k=1}^N \int_0^\8 (v_kt^{-1}_k+w_kt_k)t_k^{-2s} e^{-y\pi \left ( v_kt^{-1}_k \right )^2-y\pi \left (w_kt_k \right )^2} \frac{dt_k}{t_k}.
\end{align}
For a nonzero $a \in \R^\times$ we define
\begin{align}
k_s(a) \coloneqq e^{2 \pi a}\vert a \vert^{\frac{1}{2}+s} \left ( K_{\frac{1}{2}+s}(2\pi\vert a\vert)+\sgn(a)K_{\frac{1}{2}-s}(2\pi\vert a\vert) \right )
\end{align}
where $K_s$ is the $K$-Bessel function
\begin{align}
K_s(a) \coloneqq \int_0^\8e^{-a(t+t^{-1})/2}t^s\frac{dt}{t}.
\end{align}
We have
\begin{align}
c_{\chi,\cfrak}(\nu,2s)= 2^{-N}y^{-Ns} \prod_{k=1}^Nk_s(y\nu_k) \sum_{[\afrak] \in \Cl(F)^+} \chi(\afrak) N(\afrak)^{2s}  \sum_{\substack{(v,w) \in \afrak\cfrak \times \afrak^{-1} \dfrak^{-1}/\Ocal^{\times,+} \\ \nu=vw }}  \frac{\sgn(N(v))}{\vert N(v)\vert^{2s}}.
\end{align}
 Moreover, using the exact sequence
\begin{align}
1 \longrightarrow \Ocal^{\times}/\Ocal^{\times,+} \longrightarrow F^\times/F^{\times,+} \longrightarrow \Cl(F)^+\longrightarrow \Cl(F)\longrightarrow 1
\end{align}
we can rewrite the previous term as 
\begin{align}
& \sum_{[\afrak] \in \Cl(F)^+} \chi(\afrak) N(\afrak)^{2s}  \sum_{\substack{(v,w) \in \afrak\cfrak \times \afrak^{-1} \dfrak^{-1}/\Ocal^{\times,+} \\ \nu=vw }}  \frac{\sgn(N(v))}{\vert N(v)\vert^{2s}} \\
& = [F^\times: F^{\times,+}] \sum_{[\afrak] \in \Cl(F)} \chi(\afrak) N(\afrak)^{2s}  \sum_{\substack{(v,w) \in \afrak\cfrak \times \afrak^{-1} \dfrak^{-1}/\Ocal^{\times} \\ \nu=vw }}  \frac{\sgn(N(v))}{\vert N(v)\vert^{2s}}  \\
& = 2^N \sigma_{\chi,\cfrak}(\nu,2s).
\end{align}
where for $\nu \in F^\times$ we set
\begin{align}
 \sigma_{\chi,\cfrak}(\nu,2s) \coloneqq \sum_{\substack{ \nfrak \mid \nu \dfrak \\ \cfrak \mid \nfrak}} \frac{\chi(\nfrak)}{\N(\nfrak)^{2s}}.
\end{align}
Note that $\sigma_{\chi,\cfrak}(\nu,2s)=0$ if $\nu \not \in \cfrak\dfrak^{-1}$. 

\begin{prop}\label{Fourier expansion s} The Eisenstein series $\Ecal_{\chi,\cfrak}(\tau,s)$ has the Fourier expansion
\begin{align}
\Ecal_{\chi,\cfrak}(\tau,s) =2^{-N}y^{Ns}\chi(\cfrak)\N(\cfrak)^{-2s}\Lambda(\chi,2s)&+2^{-N}y^{-Ns}\chi(\dfrak)\N(\dfrak)^{-2s}\Lambda(\chi,-2s) \\
&  \qquad \qquad +y^{-Ns}\sum_{\substack{\nu \in \dfrak^{-1}\\ \nu \neq 0}} \sigma_{\chi,\cfrak}(\nu,2s) \prod_{k=1}^N k_s(y \nu_k) e(\tau\tr(\nu)).
\end{align}
\end{prop}
At $s=0$, we have $k_0(a)=1+\sgn(a)$. Thus, the Fourier expansion is
\begin{align} \label{fourier expansion at zero}
\Ecal_{\chi,\cfrak}(\tau)=(\chi(\cfrak)+\chi(\dfrak))L(\chi,0)+ 2^{N}\sum_{\substack{\nu \in \dfrak^{-1}\\ \nu \gg 0}} \sigma_{\chi,\cfrak}(\nu,0)e(\tau \tr(\nu)).
\end{align}
The modular curve $Y_0(p)$ has two cusps $0$ and $\8$. Let $\gamma_0=\begin{psmallmatrix}0 & -1 \\ 1 & 0	\end{psmallmatrix} \in \SL_2(\Z)$, which sends $\8$ to the cusp $0$.  A direct computation shows that
\begin{align}
\Ecal_{\chi,\cfrak}(\tau,s) \vert_{\gamma_0}=\chi(\cfrak)\N(\cfrak)^{-1-2s}\Ecal_{\chi,\cfrak^{-1}}(\tau,s),
\end{align}
which gives a Fourier expansion of $\Ecal_{\chi,\cfrak}(\tau,s)$ at the cusp $0$. If we let $\bar{\cfrak} \subset \Ocal$ be the ideal such that $p=\cfrak\bar{\cfrak}$, then
\begin{align}
\Ecal_{\chi,\cfrak}(\tau,s) \vert_{\gamma_0}=p^{Ns}\chi(\cfrak)\N(\cfrak)^{-1-2s}\Ecal_{\chi,\bar{\cfrak}}\left (\frac{\tau}{p},s \right ),
\end{align}

\subsection{The functional equation}

\begin{prop}
The Eisenstein series $\Ecal_{\chi,\cfrak}(\tau,s) = C(s) E_{\chi,\cfrak}(\tau,s)$ satisfies the functional equation
\begin{align}
\Ecal_{\chi,\cfrak}(\tau,-s)=\chi(\dfrak\cfrak)N(\dfrak\cfrak)^{2s}\Ecal_{\chi,\cfrak}(\tau,s).
\end{align}
In particular, if $\chi(\cfrak)=-\chi(\dfrak)$ then $\Ecal_{\chi}(\tau)=0$.
\end{prop}

\begin{proof} The functional equation can be deduced from Theorem \ref{functional equation lift}, but can even easier be seen from the Fourier expansion of $\Ecal_{\chi,\cfrak}(\tau,s)$. Indeed, the Fourier coefficient all satisfy 
\begin{align} \label{funceqdivisor}
c_{\chi,\cfrak}(\nu,-2s)=\chi(\dfrak\cfrak)N(\dfrak\cfrak)^{2s}c_{\chi,\cfrak}(\nu,2s).
\end{align}
For the constant term this can be seen directly, and for the higher coefficients it follows from the integral representation that replacing $s$ by $-s$ has the same effect as swaping $v$ and $w$. The functional equation then follows from reordering the sum over the narrow class group.
\end{proof}

\subsection{The derivative} Let $\cfrak$ be an ideal dividing $p$ and such that $\chi(\cfrak)=-\chi(\dfrak)$. The Eisenstein series $\Ecal_{\chi,\cfrak}(\tau,s)$ vanishes at $s=0$ and instead we can consider its derivative
\begin{align}
\Ecal'_{\chi,\cfrak}(\tau) \coloneqq \dds \Ecal_{\chi,\cfrak}(\tau,s).
\end{align}
The Eisenstein series $\Ecal_{\chi,\cfrak}(\tau,s)$ is a non-holomorphic modular form of weight $N$ for any $s \in \C$, hence $\Ecal'_{\chi,\cfrak}(\tau)$ is as well. The function $\Lambda(\chi,s)$ satisfies $\Lambda(\chi,0)=L(\chi,0)$ and
\begin{align}
\Lambda'(\chi,0)=L'(\chi,0)-\frac{N}{2}\log(4 \pi e^\gamma) L(\chi,0)
\end{align}
where $\gamma \coloneqq -\Gamma'(1)$ is the Euler-Mascheroni constant. From the Fourier expansion of $\Ecal_{\chi,\cfrak}(\tau,s)$ we deduce that 
\begin{align}\label{fourier expansion derivative}
\Ecal'_{\chi,\cfrak}(\tau)= \log(\alpha(\chi,\cfrak))+A_{\chi}+B_{\chi}\log(y)+& 2^{N+1}\sum_{\substack{\nu \in \dfrak^{-1} \\ \nu \gg 0}}\sigma'_{\chi,\cfrak}(\nu,0)e(\tau\tr(\nu)) \\
& + 2^{N-1}\sum_{l=1}^N\sum_{\substack{\nu \in \dfrak^{-1} \\ \nu_l<0 \\ \nu_k>0, \ k \neq l}}\sigma_{\chi,\cfrak}(\nu,0) \left ( \dds k_s(-y \vert \nu_l \vert) \right )e(\tau\tr(\nu)),
\end{align}  where the coefficients are
\begin{align}
\log(\alpha(\chi,\cfrak)) & = -2^{2-N}\chi(\dfrak) \log(\N(\cfrak \dfrak))L(\chi,0) \in \log(\Qbar),\\
A_{\chi}  & = 2^{2-N}\chi(\dfrak)\Lambda'(\chi,0), \\
B_{\chi} & = 2^{1-N}N \chi(\dfrak)  L(\chi,0) \in \Q, \\
\sigma'_{\chi,\cfrak}(\nu,0) & = - \sum_{\substack{ \nfrak \mid \nu \dfrak \\ \cfrak \mid \nfrak}} \chi(\nfrak)\log(\N(\nfrak)) \in \log(\Qbar)
\end{align}
Note that $\log(\alpha(\chi,\cfrak))$ depends on $\cfrak$ and is the logarithm of an algebraic number, since $L(\chi,0)$ is rational by Shintani's theorem \ref{Shintani}. On the other hand, the term $A_{\chi}$ is transcendental and does not depend on $\cfrak$.
\begin{rmk} In \eqref{fourier expansion derivative}, the higher Fourier coefficients are the value at $s=0$ of
\begin{align}
\sum_{l=1}^N\sum_{\substack{\nu \in \dfrak^{-1} \\ \nu \neq 0}} \sigma_{\chi,\cfrak}(\nu,2s)k_s(y\nu_1) \cdots \left ( \frac{\partial}{\partial s}k_s(y\nu_l) \right )  \cdots k_s(y\nu_N).
\end{align}
Since $k_0(a)=0$ if $a<0$, it follows that the inner sum is only nonzero if $\nu_k>0$ for $k\neq l$. Finally, if $\nu_l>0$ then $\nu \gg 0$ and $\sigma_{\chi,\cfrak}(\nu\dfrak,0)=0$ by the functional equation \eqref{funceqdivisor}.
\end{rmk}
Recall that 
\begin{align}
\beta_{s}(a)=\int_1^\8e^{-at}t^{s-1}dt.
\end{align}

\begin{lem} We have \begin{align}
\left ( \dds k_s(-a) \right )=  \beta_{0}(4\pi a ).
\end{align}
\end{lem}
\begin{proof}
By the computations in \cite[Proposition.~3.3 p.~278]{gz86}, one finds that
\begin{align} \label{equalitykV}
k_s(a)= \frac{i}{2}e^{-2\pi a} \frac{\Gamma(1+s)}{\pi^{1+s}} V_s(a)
\end{align}
where
\begin{align}
V_s(a)=\int_{-\8}^\8 \frac{e^{-2i \pi ax}}{(x+i)(x^2+1)^s}dx \qquad \re(s)>0.
\end{align}
Moreover, by  \cite[Proposition.~3.3 e)]{gz86}, we have for $a>0$
\begin{align}
\dds V_s(-a)=-2i\pi e^{2\pi a} \int_1^\8e^{-4\pi a t} \frac{dt}{t} = -2i\pi e^{2\pi a}\beta_{0}(4\pi a ).
 \end{align}
Since $k_0(a)=1+\sgn(a)$, it follows from \eqref{equalitykV} that $V_0(-a)$ for $a>0$. Thus, for $a>0$ we have
\begin{align}
\kappa(a)& = \dds k_s(-a) = \beta_{0}(4\pi a ).
\end{align}
\end{proof}

\begin{prop} \label{Gchifourier}The derivative $\Ecal'_{\chi,\cfrak}(\tau)$ is a non-holomorphic modular form of weight $N$ for $\Gamma_0(p)$ with Fourier expansion
\begin{align}
\Ecal'_{\chi,\cfrak}(\tau)= \log(\alpha(\chi,\cfrak))+A_{\chi}+B_{\chi}\log(y)-\sum_{n=1}^\8\log(J_{\chi,\cfrak}(n))e(n \tau)  + \sum_{n \in \Z} a_{\chi,\cfrak}(n,y)e(n \tau)
\end{align}
where
\begin{align}
J_{\chi,\cfrak}(n) \coloneqq \sum_{\substack{\nu \in \dfrak^{-1} \\ \nu \gg 0 \\ \Tr(\nu)=n}}\sum_{\substack{\nfrak \mid \nu \dfrak \\ \cfrak \mid \nfrak }} \N(\nfrak)^{\chi(\nfrak)2^{N+1}} \in \Q, \qquad 
a_{\chi,\cfrak}(n,y) \coloneqq 2^{N-1}\sum_{l=1}^N \sum_{\substack{\nu \in \dfrak^{-1} \\ \Tr(\nu)=n \\ \nu_l<0 \\ \nu_k>0, \ k \neq l}}\sigma_{\chi,\cfrak}(\nu,0)\beta_0(4\pi y \vert \nu_l \vert).
\end{align}
\end{prop}
Since $\chi(\cfrak)=\chi(\bar{\cfrak})$, the Eisenstein series $\Ecal_{\chi,\bar{\cfrak}}\left (\tau,s \right )$ also vanishes at $s=0$. From 
\begin{align}
\Ecal_{\chi,\cfrak}(\tau,s) \vert_{\gamma_0}=p^{Ns}\chi(\cfrak)\N(\cfrak)^{-1-2s}\Ecal_{\chi,\bar{\cfrak}}\left (\frac{\tau}{p},s \right ),
\end{align} it follows that
\begin{align}
\Ecal'_{\chi,\cfrak}(\tau)\vert_{\gamma_0}=\frac{\chi(\cfrak)}{\N(\cfrak)}\Ecal'_{\chi,\bar{\cfrak}}\left (\frac{\tau}{p} \right ).
\end{align}

When $N=2$ the computation are as in \cite{gzsing}, with a slight difference for the constant term (the two expressions are related by the functional equation). The sum appearing in the definition of $J_{\chi,\cfrak}(n)$ is finite. For $\nu \in \dfrak^{-1}$, let $w_\nu \in \R^N$ be the vector whose entries are the $N$ embeddings of $\nu$. The set $L$ of all vectors $w_\nu$ is a lattice in $\R^N$. Requiring that $\nu_i=\langle w_\nu,e_i \rangle$ is positive for all $1 \leq i \leq N$ means that we only consider the vectors $w_\nu$ that are in the quadrant spanned by the vector $e_1, \dots,e_N$. On the other hand, requiring that 
\begin{align}
\tr(\nu)=\langle w_\nu,e_1+\cdots+e_N \rangle=n
\end{align} means that we only take the lattice vectors on a hyperplane that intersects the quadrant in a compact set. 

On the other hand, the sum appearing in $a_{\chi,\cfrak}(n,y)$ is infinite, but converges due to the following Lemma and the fact that $\beta_0(t)=O(e^{-Ct})$.

\begin{lem} \label{convergence lemma} Let $f \colon \R \longrightarrow \R$ be a function such that $f(x)=O(\vert x \vert^{-N-\epsilon})$ for some $\epsilon>0$. Then
\begin{align}
\sum_{\substack{\nu \in \dfrak^{-1} \\ \Tr(\nu)=n \\ \nu_l<0 \\ \nu_k>0, \ k \neq l}}f(\vert \nu_l \vert)
\end{align}
converges absolutely.
\end{lem}
\begin{proof}
Assume without of loss of generality that $l=N$. Let $L=\{ \sigma(\nu) \in \R^N \ \vert \ \nu \in \dfrak^{-1} \}$ be the lattice as above. The subset
\begin{align}
L_{<0} \coloneqq \{ w=\sigma(\nu) \in L \ \vert \  w_1, \dots,w_{N-1}>0, \ w_N<0, \ \tr(\nu)=n\}
\end{align}
is the intersection of the hyperplane $\tr(\nu)=n$ with the quadrant spanned by the vector $e_1, \dots,e_{N-1}$ and $-e_N$ (instead of $e_N$). Since the first $N-1$ embeddings are positive, we have
\begin{align}
n=\tr(\nu)=\vert \nu_1 \vert+ \cdots+\vert \nu_{N-1} \vert-\vert \nu_N \vert =\lVert w_\nu\rVert_1-2\vert \nu_N \vert,
\end{align}
where $\lVert w \rVert_1$ is the $L^1$-norm. For $d \geq 1$ let $r(d)$ be the number of vectors in $L_{<0}$ that satisfy $d \leq \lVert w \rVert_1 < (d+1)$. If we denote by $m \coloneqq \min_{w \in L} \lVert w \rVert_1$ the length of the shortest vector in $L$, then for each vector $w$ there is a unique integer $d$ such that $ dm \leq \lVert w \rVert_1 < m(d+1)$. Thus, the sum can be bounded by
\begin{align}
\sum_{d>n/m} \frac{r(dm)}{(dm-n)^{N+\epsilon}}.
\end{align}
Since $r(d)=O(d^{N-1})$, this proves the claim.
\end{proof}

Let $f \in H_{2-N}(\Gamma_0(p))$ be a harmonic weak Maass form and
\begin{align}
f \vert_{\gamma_0}(\tau)& = \sum_{\substack{n \in \frac{1}{p}\Z \\ n \gg -\8}}a^+_0(f,n)e\left (n\tau \right ) + \sum_{n \in \frac{1}{p}\NN} a^-_0(f,n)\beta_k(n,y)e\left (-n\tau \right ) .
\end{align}
a Fourier expansion at $0$. Note that the width $w_0$ of the cusp $0$ is $p$.
\begin{thm}\label{main formula}Let $f \in H_{2-N}(\Gamma_0(p))$ be a harmonic weak Maass form and $g=\xi_{2-N}(f) \in S_N(\Gamma_0(p))$. Suppose that the principal parts of $f$ at both cusps have rational coefficients. Then
\begin{align}
\int_{c_\epsilon}\Phi(z_1,f,\phi_{\chi,\cfrak}) = \log(\alpha(f,\chi,\cfrak))-\frac{1}{2N} \langle \Ecal'_{\chi,\cfrak},g \rangle -\frac{A_{\chi}}{2N} \left (a^+_\8(f,0)+p \frac{\chi(\cfrak)}{\N(\cfrak)} a^+_0(f,0) \right ),
\end{align}
for some algebraic number $\alpha(f,\chi,\cfrak)$.
\end{thm}
\begin{proof}  From Proposition \ref{Gchifourier} we have 
\begin{align}
\Ecal'_{\chi,\cfrak}(\tau)&= \log(\alpha(\chi,\cfrak))+A_{\chi}+B_{\chi}\log(y)+2^{N+1} \sum_{n=1}^\8\log(J_{\chi,\cfrak}(n))e(n \tau)  + \sum_{n \in \Z} a_{\chi,\cfrak}(n,y)e(n \tau) \\
\restr{\Ecal'_{\chi,\cfrak}}{\gamma_0}(\tau) &= \frac{\chi(\cfrak)}{\N(\cfrak)} \left ( \log(\alpha(\chi,\bar{\cfrak}))+A_{\chi}+B_{\chi}\log(y/p)+ \sum_{n=1}^\8\log(J_{\chi,\bar{\cfrak}}(n))e\left ( \frac{n\tau}{p} \right )  + \sum_{n \in \Z} a_{\chi,\bar{\cfrak}}(n,y)e\left ( \frac{n\tau}{p} \right ) \right )
\end{align}
since $\restr{\Ecal'_{\chi,\cfrak}}{\gamma_0}(\tau)=\frac{\chi(\cfrak)}{\N(\cfrak)}\Ecal'_{\chi,\bar{\cfrak}}\left (\frac{\tau}{p} \right )$. From the proof of Theorem \ref{main formula 1} it follows that
\begin{align}
-2N\int_{c_\epsilon}\Phi(z_1,f,\phi) =  \langle \Ecal'_{\chi,\cfrak},g \rangle+\sum_{ r \in \Gamma_0(p) \backslash \PP^1(\Q)}w_r \kappa_r(c,f,\phi_{\chi,\cfrak})
\end{align}
where the width at the cusp $0$ is $w_0=p$, and
\begin{align}
\kappa_\8(f,\chi,\cfrak)& =\left ( \log(\alpha(\chi,\cfrak)) +A_{\chi} \right )a^+_\8(f,0)  + \lim_{T \rightarrow \8}\sum_{n \geq 1} \left (\log(J_{\chi,\cfrak}(n))+a_{\chi,\cfrak}(n,T)\right )  a^+_\8(f,-n), \\
\kappa_0(f,\chi,\cfrak)& =\frac{\chi(\cfrak)}{\N(\cfrak)} \left ( \log(\alpha(\chi,\bar{\cfrak})) +A_{\chi} -B_\chi \log(p) \right )a^+_0(f,0)   \\
& \hspace{4cm} + \frac{\chi(\cfrak)}{\N(\cfrak)} \lim_{T \rightarrow \8}\sum_{n \geq 1} \left ( \log(J_{\chi,\bar{\cfrak}}(pn))+a_{\chi,\bar{\cfrak}}(pn,T)\right )  a^+_0(f,-n).
\end{align}

Since $\beta_0(4\pi T \vert \nu_l \vert)=O(e^{-\epsilon T})$ as $T \rightarrow \8$, we have
\begin{align}
\lim_{T \rightarrow \8} a_{\chi,\cfrak}(n,T) =\lim_{T \rightarrow \8}\sum_{l=1}^N \sum_{\substack{\nu \in \dfrak^{-1} \\ \Tr(\nu)=n \\ \nu_l<0 \\ \nu_k>0, \ k \neq l}}\sigma_{\chi,\cfrak}(\nu,0)\beta_0(4\pi T \vert \nu_l \vert)=0.
\end{align}
Thus, the contribution from $\8$ is 
\begin{align}
\kappa_\8(f,\chi,\cfrak) =A_{\chi}a^+_\8(f,0) + \log(\alpha_\8(f,\chi,\cfrak))
\end{align}
where
\begin{align}
\alpha_\8(f,\chi,\cfrak)=\alpha(\chi,\cfrak)^{a^+_\8(f,0)}\prod_{n \geq 1}  J_{\chi,\cfrak}(n)^{ a^+_\8(f,-n)}
\end{align}
and the contribution from $0$ is
\begin{align}
\kappa_0(f,\chi,\cfrak) =\frac{\chi(\cfrak)}{\N(\cfrak)} A_{\chi} a^+_0(f,0)  + \frac{\chi(\cfrak)}{\N(\cfrak)}\log(\alpha_0(f,\chi,\cfrak))
\end{align}
where
\begin{align}
\alpha_0(f,\chi,\cfrak)=\alpha(\chi,\bar{\cfrak})^{a^+_0(f,0)}p^{-B_\chi a^+_0(f,0)}\prod_{n \geq 1}  J_{\chi,\bar{\cfrak}}(pn)^{ a^+_0(f,-n)}.
\end{align}
Thus, we get
\begin{align}
\kappa_\8(f,\chi,\cfrak)+p\kappa_0(f,\chi,\cfrak)= A_{\chi} \left (a^+_\8(f,0)+p \frac{\chi(\cfrak)}{\N(\cfrak)} a^+_0(f,0) \right ) +\log(\alpha(f,\chi,\cfrak)),
\end{align}
where
\begin{align}
\alpha(f,\chi,\cfrak) \coloneqq \alpha_\8(f,\chi,\cfrak) \alpha_0(f,\chi,\cfrak)^{p\chi(\cfrak)\N(\cfrak)^{-1}} \in \Qbar.
\end{align}
\end{proof}
%

\printbibliography
\end{document}